\documentclass[11pt, notitlepage]{article}
\usepackage{amssymb,amsmath,comment}
\catcode`\@=11 \@addtoreset{equation}{section}
\def\thesection{\arabic{section}}

\def\theequation{\thesection.\arabic{equation}}
\catcode`\@=12
\usepackage{colortbl}%
\usepackage{a4wide}
\usepackage{parskip}

\newcommand{\ds} {\displaystyle}
\newcommand{\e}{\epsilon}
\newcommand{\pa} {\partial}
\newcommand{\al} {\alpha}
\newcommand{\ba} {\beta}
\newcommand{\vsg} {\varsigma}
\newcommand{\de} {\delta}
\newcommand{\ga} {\gamma}
\newcommand{\Ga} {\Gamma}
\newcommand{\Om} {\Omega}
\newcommand{\ra} {\rightarrow}
\newcommand{\De} {\Delta}
\newcommand{\la} {\lambda}
\newcommand{\La} {\Lambda}
\newcommand{\sg} {\sigma}
\newcommand{\noi} {\noindent}
\newcommand{\na} {\nabla}
\newcommand{\ul} {\underline}
\newcommand{\ov} {\overline}
\newcommand{\mb} {\mathbb}
\newcommand{\mc} {\mathcal}

\newcommand{\tl} {\tilde}

\usepackage[all]{xy}
\catcode`\@=11
\def\theequation{\@arabic{\c@section}.\@arabic{\c@equation}}
\catcode`\@=12

\def\QED{\hfill {$\square$}\goodbreak \medskip}

\newtheorem{Theorem}{Theorem}[section]
\newtheorem{Lemma}[Theorem]{Lemma}
\newtheorem{Proposition}[Theorem]{Proposition}
\newtheorem{Corollary}[Theorem]{Corollary}
\newtheorem{Remark}[Theorem]{Remark}
\newtheorem{Definition}[Theorem]{Definition}

\def\XXint#1#2#3{{\setbox0=\hbox{$#1{#2#3}{\int}$ }
		\vcenter{\hbox{$#2#3$ }}\kern-.6\wd0}}

\begin{document}
{\vspace{0.01in}

	\title
{Sobolev and H\"older regularity results for some singular nonhomogeneous quasilinear problems}
\author{  Jacques Giacomoni$^{\,1}$ \footnote{e-mail: {\tt jacques.giacomoni@univ-pau.fr}}, \ Deepak Kumar$^{\,2}$\footnote{e-mail: {\tt deepak.kr0894@gmail.com}},  \
	and \  K. Sreenadh$^{\,2}$\footnote{
		e-mail: {\tt sreenadh@maths.iitd.ac.in}} \\
	\\ $^1\,${\small Universit\'e  de Pau et des Pays de l'Adour, LMAP (UMR E2S-UPPA CNRS 5142) }\\ {\small Bat. IPRA, Avenue de l'Universit\'e F-64013 Pau, France}\\  
	$^2\,${\small Department of Mathematics, Indian Institute of Technology Delhi,}\\
	{\small	Hauz Khaz, New Delhi-110016, India } }

\date{}

\maketitle

\begin{abstract}
 \noi This article deals with the study of
  the following singular quasilinear equation: 
  \begin{equation*}
   (P) \left\{ \ -\Delta_{p}u -\Delta_{q}u  = f(x) u^{-\delta},\; u>0 \text{ in }\; \Om; \; u=0 \text{ on }  \pa\Om,
   	\right. 
  \end{equation*}
  where $\Om$ is a bounded domain in $\mathbb{R}^N$ with $C^2$ boundary $\pa\Om$, $1< q< p<\infty$, $\de>0$ and $f\in L^\infty_{loc}(\Om)$ is a non-negative function which behaves like $\textnormal{dist}(x,\pa\Om)^{-\ba},$ $\ba\ge 0$ near the boundary of $\Om$. We prove the existence of a weak solution in $W^{1,p}_{loc}(\Om)$ and its behaviour near the boundary for $\ba<p$. Consequently, we obtain optimal Sobolev regularity of weak solutions.  By establishing the comparison principle,  we prove the uniqueness of weak solution for the case $\ba<2-\frac{1}{p}$. Subsequently, for the case $\ba\ge p$, we prove the non-existence result. Moreover, we prove H\"older regularity of the gradient of weak solution to a more general class of quasilinear equations involving singular nonlinearity as well as lower order terms (see \eqref{Prb}). This result is completely new and of independent interest.
  In addition to this, we prove H\"older regularity of minimal weak solutions of $(P)$ for the case $\beta+\delta\geq 1$ that has not been fully answered in former contributions even for $p$-Laplace operators.
 \medskip

 \noi \textbf{Key words:} $(p,q)$-Singular equations, nonhomogeneous operators, existence and uniqueness results, comparison principle, non-existence results, regularity results.

\medskip

\noi \textit{2010 Mathematics Subject Classification:} 35J20, 35J62, 35J75, 35J92, 35B65, 49N60

\end{abstract}

\bigskip

\section {Introduction}
 \noindent The purpose of this article is to study the existence and regularity of the weak solution to the following prototype singular problem:
    \begin{equation*}
     (P) \; \;\left\{\begin{array}{rllll}
     -\Delta_{p}u-\Delta_{q}u & = f(x) u^{-\de}, \;  u> 0 \text{ in } \Om \\ u&=0 \quad \text{ on } \partial\Om,
     \end{array}
     \right.
     \end{equation*}
 \noi where $\Om$ is a bounded domain in $\mb R^N$ with $C^2$ boundary $\pa\Om$, $1< q< p<\infty$ and $\de>0$.
 $\Delta_{p}$ is the $p$-Laplace operator, defined as $\De_p u:= \textnormal{div} (|\na u|^{p-2}\na u)$. The operator $A_{p,q}:=-\Delta_{p}-\Delta_{q}$ is known as $(p,q)$-Laplacian which arises while studying the stationary solutions of general reaction-diffusion equation
 \begin{equation}\label{prb1}
 u_t= \textnormal{div}[A(u)\nabla u]+ r(x,u),
 \end{equation}
 where $A(u)= |\nabla u|^{p-2}+|\nabla u|^{q-2}$. The 
 problem \eqref{prb1} has  applications in biophysics, plasma physics and chemical reactions, with double phase features, where the function $u$ corresponds to the concentration term, the first term on the right side represents diffusion with a diffusion coefficient $A(u)$ and the second term is the reaction which relates to sources and loss processes. For more details,  readers are referred to \cite{marano} and its references.\par 
The energy functional of equations driven by the $(p,q)$-Laplacian falls in the category of  the so-called functionals with
 nonstandard growth conditions of $(p, q)$-type, according to Marcellini's terminology \cite{marce2}. These kinds of functionals involve  integrals of the form
 \begin{align*}
 	I(u)=\int_{\Om} h(x,\na u(x))~dx,
 \end{align*} 
  where the energy density, $h$, satisfies
 $$|\xi|^p\leq |h(x,\xi)|\leq  |\xi|^q+1,\quad 1\leq p\leq q.$$
 The physical significance of these models lies in the field of nonlinear elasticity, specifically in homogenisation theory. A particular form of the above class of functionals is the double phase functional given by 
 $$
 u\mapsto\int_{\Omega} (|\nabla u|^p+a(x)|\nabla u|^q)dx,\quad 0\leq a(x)\leq L,\ 1<p<q.
 $$
 This functional was first introduced by Zhikov in \cite{zhikov1}, to model the Lavrentiev phenomenon on strongly anisotropic materials. The study has been continued by Mingione et al. \cite{baroni,colombo} and R\v adulescu et al. \cite{prrzamp,prrpams,radopuscula}.  \par
 
 For the case $p=q$, problem $(P)$ involves the homogeneous $p$-Laplacian operator and the problem takes the following form
 \begin{equation}\label{eqb1}
 -\Delta_p u=f(x)u^{-\de}, \ \ u>0 \ \mbox{ in }\Om; \ u=0 \ \mbox{ on }\pa\Om.
 \end{equation}
 This type of equation has numerous 
 applications in the physical world such as non-newtonian flows in porous media and heterogeneous catalysts. There has been an extensive study in this direction since the pioneering work of Crandall, Rabinowitz and Tartar \cite{crandall},
  see for instance \cite{adimurthi,bocardo,diaz,ghergujfa,haitao,hernand,hirano,lazer} and references therein. In \cite{crandall}, authors studied \eqref{eqb1} for $p=2$ with $f$ as a nonnegative bounded function and $\de>0$. In this work, they proved the existence and uniqueness of solution in $C^2(\Om)\cap C(\ov\Om)$ and behaviour of the solution near the boundary is also discussed when $f=1$. Lazer and Mckenna in \cite{lazer}, considered \eqref{eqb1} when $p=2$ and $f\in C^{\al}(\ov\Om)$ is positive. Authors proved the existence of unique solution in $C^{2+\al}(\Om)\cap C(\ov\Om)$ for all $\de>0$. Moreover, they proved that the solution is not in $C^1(\ov\Om)$ if $\de>1$ and it is in $H^1_0(\Om)$ if and only if $\de<3$. Subsequently, Boccardo and Orsina considered \eqref{eqb1} when the leading differential operator takes the form $-\textnormal{div}(A(x)\na u)$, where $A$ is a bounded elliptic operator and $f$ is either a nonnegative function belonging to some Lebesgue space or a nonnegative bounded radon measure. Here they proved the existence and some Sobolev regularity results. Concerning the case, when $f$ has a singularity, Diaz, Hern\'andez and Rakotoson \cite{diaz} considered the case where $f$ behaves as some negative power of the distance function. Here, regularity of $\na u$ in Lorentz spaces is proved. Furthermore, for the singular problem with $p=2$ having subcritical or critical growth perturbation with respect to the Sobolev embedding, we mention the contribution of Haitao in \cite{haitao}, and Hirano, Saccon and Shioji in \cite{hirano}. For the lower dimension case ($N=2$), 
 Adimurthi and Giacomoni \cite{adimurthi} considered problem \eqref{eqb1} with $0<\de<3$ and a perturbation of critical growth with respect to the Trudinger-Moser embeddings. For a thorough analysis of semilinear elliptic equations with singular nonlinearities we refer to the monograph by Ghergu and R\u adulescu \cite{ghergu} and an overview article by Hern\'andez, Mancebo and Vega \cite{hernand1}.\par
 For the quasilinear case, that is $p\neq 2$, Giacomoni and Sreenadh \cite{giacsree} studied \eqref{eqb1} with a perturbation having $p-1$ superlinear growth and $f(x)=\la$, a real parameter. Authors proved that there exists a weak solution in $W^{1,p}_0(\Om)\cap C(\ov\Om)$, for small $\la>0$, if and only if $\de<2+1/(p-1)$. Subsequently, 
 Giacomoni, Schindler and Tak\'a\v{c} \cite{giacomoni} studied \eqref{eqb1} with subcritical and critical perturbation with respect to Sobolev embedding for the case $0<\de<1$ and $f(x)=\la$. 
 Using variational methods, authors proved existence of multiple solutions in $C^{1,\al}(\ov\Om)$, for some $\al\in(0,1)$. Here global multiplicity of solutions is also proved with respect to the parameter $\la$. Thereafter, Canino, Sciunzi and Trombetta \cite{canino}, and Bougherara, Giacomoni and Hernandez \cite{bough} studied problem \eqref{eqb1} under different summability conditions on $f$. Under the assumption that $f\in L^1(\Om)$, Canino et al. in \cite{canino}, proved the existence result and with higher integrability assumptions, they also obtained the uniqueness result. 
 While in \cite{bough}, authors considered the further case $f\in L^\infty_{loc}(\Om)$, a nonnegative function which behaves like $dist(x,\pa\Om)^{-\ba}$ near the boundary $\pa\Om$, for $\ba\ge 0$. Exploiting the method of sub and supersolution authors proved the existence of a solution for all $\de>1-p$ and $\ba\in[0,p)$. In this work, behaviour near the boundary and Sobolev regularity of the solution are also discussed. Moreover, authors proved the uniqueness result when $1-p<\de<2-\ba+\frac{1-\ba}{p-1}$. For the case of $p=1$, Cicco, Giachetti, Oliva and Petitta \cite{decicco} studied \eqref{eqb1} with the nonlinear term $f(x)h(u)$, where $h$ has a singularity at $0$. Under certain assumptions on $f$ and $h$, authors proved the existence, uniqueness and regularity result.\par
As far as the equations  with nonhomogeneous operators involving singular nonlinearity are concerned, we would like to draw the attention of readers towards the recent works Kumar, R\v adulescu and Sreenadh \cite{dpk} and Papageorgiou, R\v adulescu and Repov\v s \cite{papageorg}. In \cite{dpk}, authors consider $(p,q)$-Laplace equation of type $(P)$ with critical growth perturbation with respect to the Sobolev embedding and $f(x)=\la$. Splitting the Nehari manifold authors proved the existence of at least two positive solutions and the  $L^\infty$ estimates. Furthermore, they obtained the global existence result using Perron's method. Whereas in \cite{papageorg}, authors considered the following equation
\begin{align*}
	-\textnormal{div} A(\na u)= \la\nu(u)+f(x,u) \ \ u>0 \ \mbox{ in }\Om; \ u=0 \ \mbox{ on }\pa\Om,
\end{align*}
where $A:\mb R^N\to \mb R^N$ satisfies certain structure condition, $\nu(u)$ behaves like $u^{-\de}$ for $\de\in(0,1)$ and $f$ is a Carath\'eodory function with subcritical growth. Here authors proved the existence of $\la_*>0$ such that the problem has at least two solutions if $\la<\la_*$, at least one solution for $\la=\la_*$ and no solution for $\la>\la_*$. The Sobolev and H\"older regularity were still open questions after these works.\par
Regarding this issue, we consider the following general form of problem $(P)$,
\begin{align}\label{eqb2}
	-\text{div} A(x,u,\na u)=B(x,u,\na u)+g(x) \quad\mbox{in }\Om.
\end{align} 
 We first mention the work of Lieberman \cite{liebm88,liebm91} for solution of quasilinear elliptic equations with nonsingular nonlinearity. In \cite{liebm88}, the author proved that weak solutions of \eqref{eqb2} are in $C^{1,\al}(\ov\Om)$ when $A$ and $B$ satisfy the following structure condition: 
  \begin{equation}\label{lbst}
  \begin{aligned}
    \la(\kappa+|z|)^{p-2}|\xi|^2\le a^{ij}(x,u,z)\xi_i \xi_j\le \La (\kappa+|z|)^{p-2}|\xi|^2 \quad\mbox{and}\\
     |B(x,u,z)|\le \La(\kappa+|z|)^{p}, \quad \mbox{for } (x,u,z)\in\Om\times\mb R\times\mb R^N,
  \end{aligned}
  \end{equation} 
  with $g=0$, where $a^{ij}=\pa A^i/\pa z_j$, $0<\la\le\La$ and $\kappa\in[0,1]$. Subsequently, in \cite{liebm91} interior H\"older continuity result is established for the gradient of solution to \eqref{eqb2} when $A$ and $B$ satisfies structure condition involving special kind of Orlicz functions (including $p,q$ type growth) and $g=0$. Concerning the quasilinear equations with singular nonlinearity, Giacomoni, Schindler and Tak\'a\v{c} \cite{giacomoni,giacomoni2} obtained H\"older continuity results for weak solutions to \eqref{eqb2} when $A(x,u,\na u)=A(x,\na u)$ and $B=0$. In \cite{giacomoni}, following the approach of \cite{liebm88}, authors proved that the weak solution, which behaves like the distance function near the boundary, is in $C^{1,\al}(\ov\Om)$, for some $\al\in(0,1)$, when $0\le g\le C d(x)^{-\sg}$ for $\sg<1$ and $d(x):=dist(x,\pa\Om)$. In the latter work \cite{giacomoni2}, authors proved that solution $u\in W^{1,p}_0(\Om)$, is in $C^{0,\al}(\ov\Om)$ when $0\le u\le C d(x)^{\sg'}$ and $0\le g(x) \le C d(x)^{-\sg}$ with $0<\sg'<\sg<\sg'+1$. \par
Inspired from above discussion, in this work we consider singular problems driven by the nonhomogeneous $(p,q)$-Laplace operator. We assume that $f\in L^\infty_{loc}(\Om)$ satisfies the following condition: 
 \begin{align}\label{eqf}
 	c_1 \; d(x)^{-\ba} \le f(x) \le c_2 \; d(x)^{-\ba} \quad \mbox{in } \Om_\varrho,
 \end{align}
 where $c_1,c_2$ are nonnegative constants, $\ba\ge 0$ and $\Om_\varrho:= \{ x\in\ov\Om : d(x)<\varrho \}$ for $\varrho>0$. To prove the existence of a weak solution, we perturb the problem $(P)$ by taking the nonlinear term as $f_\e\le f$, which does not contain any singularity, and replacing $u^{-\de}$ by $(u+\e)^{-\de}$. Thus the standard Schauder fixed point theory and the elliptic regularity theory can be applied to get the existence of a unique  solution $u_\e\in C^{1,\al}(\ov\Om)$ (see Lemma 2.1). By establishing comparison of $u_\e$ with some function of the distance function, we prove convergence of $u_\e$ to $u$, the minimal weak solution to problem $(P)$.
 Due to the nonhomogeneous nature of the leading operator, unlike the case of $p$-Laplace equation, we can not use some scalar multiple of eigenfunctions of $-\De_p$ to obtain the suitable sub and super solution involving the distance function. To overcome this difficulty, we exploit $C^2$ regularity of the boundary $\pa\Om$, and the fact that the distance function is $C^2$ in some neighborhood of the boundary. We use this in place of the first eigenfunction of $-\De_p$ to construct a suitable sub and super solutions.
 The aforementioned behaviour of $u_\e$ near the boundary helps us to establish the optimal Sobolev regularity for the weak solution to $(P)$ obtained as the limit of $u_\e$, that is, we prove the existence of a constant $\rho_0\ge 1$ such that $u^\rho\in W^{1,p}_0(\Om)$ if and only if $\rho>\rho_0$. Another application of this boundary behaviour is that by comparison with suitable $u_\e$, we establish the non-existence result for the case of $\ba\ge p$. This result is new even for the homogeneous quasilinear elliptic operators like $p$-Laplacian. Moreover, we prove a comparison principle for sub and super solution of $(P)$ in $W^{1,p}_{loc}(\Om)$ for the case of $\ba<2-\frac{1}{p}$. Using suitablely the Hardy inequality, this result improves former contribution even for the operators like $p$-Laplacian, by considering a new notion of solutions and a larger class of weight functions $f$. In \cite{giacomoni2}, authors obtained comparison principle when the solution is in the energy space, $W^{1,p}_0(\Om)$ while Canino et al. in \cite{canino} considered the case when $f$ belongs to some Lebesgue space.  A direct consequence is the uniqueness result for the case $\ba<2-\frac{1}{p}$. \par
 Since the boundary $\pa\Om$ is $C^2$, it follows from \cite[Lemma 14.6, p. 355]{gilbarg} that there exists $\mu>0$ such that $d\in C^2(\Om_\mu)$.  Without loss of generality, we may assume $\varrho\le \min\{\frac{\mu}{2},1\}$, so that $|\De d|\in L^\infty(\Om_\varrho)$ and \eqref{eqf} also holds.
We define the notion of weak solution to $(P)$ as follows.
 \begin{Definition}\label{dfn1}
  A function $u\in W^{1,p}_{loc}(\Om)$ is said to be a weak sub-solution (resp. super-solution) of problem $(P)$ if the following holds
 	\begin{enumerate}
 	  \item[(i)] for every $K\Subset \Om$, there exists a constant $C_K>0$ such that $u\ge C_K$ in $K$,
 	  \item[(ii)] for all $\phi\in C_c^\infty(\Om)$, with $\phi\ge 0$ in $\Om$,
 	 	  \begin{align*}
 	 	  	\int_{\Om} |\na u|^{p-2} \na u\na \phi+ \int_{\Om} |\na u|^{q-2} \na u\na \phi \le( \text{ resp.}\ge) \int_{\Om} f(x)\;u^{-\de}\phi,
 	  	  \end{align*}
 	 \item[(iii)] there exists $\ga\ge 1$ such that $u^\ga \in W^{1,p}_0(\Om)$.
 	 \end{enumerate}
  A function which is both sub and super solution of $(P)$ is called a weak solution.
 \end{Definition}
 We remark that the definition of weak solution considered above is a weaker notion of solution with respect to \cite{giacomoni,giacomoni2}. Moreover, the condition $(iii)$ in the above definition appears due to lack of the trace mapping in $W^{1,p}_{loc}(\Om)$ and this also implies the following definition of the boundary datum of $u$. 
\begin{Definition}
	We say that $u\le 0$ on $\pa\Om$, if $(u-\e)^+\in W^{1,p}_0(\Om)$ for every $\e>0$ and $u=0$ on $\pa\Om$ if $u\ge 0$ in $\Om$ and $u\le 0$ on $\pa\Om$.
\end{Definition}
\begin{Definition}
 We say a weak solution $u$ of $(P)$, is in the conical shell $\mc C_{d_{\ba,\de}}$ if it is continuous and satisfies the following
 \begin{equation*}
	\left\{\begin{array}{lr}
		\eta \; d(x)\le u(x) \le \Ga \; d(x)  &\mbox{ if }\ba+\de<1, \\
		\eta d(x) \log^\frac{1}{p-\ba}\Big(\frac{L}{d(x)}\Big) \le u(x) \le \Ga d(x) \log^\frac{1}{p-\ba}\Big(\frac{L}{ d(x)}\Big) &\mbox{ if } \ba+\de=1, \\
		\eta d(x)^{\frac{p-\ba}{p-1+\de}}  \le u(x) \le \Ga  d(x)^{\frac{p-\ba}{p-1+\de}} &\mbox{ if }\ba+\de>1,
	\end{array}
	\right.
\end{equation*}
 for some positive constants $\eta,\Ga>0$ and $L>0$ is sufficiently large.
\end{Definition}
Now, we state our main existence result.
 \begin{Theorem}\label{thm1}
  Let $\ba\in[0,p)$, then problem $(P)$ admits a weak minimal solution $u\in W^{1,p}_{loc}(\Om)\cap\mc C_{d_{\ba,\de}}$, in the sense of definition \ref{dfn1}. Moreover, $u^\rho\in W^{1,p}_0(\Om)$ if and only if $\rho>\rho_0:=\frac{(p-1)(\ba+\de-1)}{(p-\ba)}>0$. Further, $u\in W^{1,p}_0(\Om)$  if and only if  $\de < 2+ \frac{1-\ba p}{p-1}$.
 \end{Theorem}
 To obtain the uniqueness result, we establish the following weak comparison principle.
\begin{Theorem}\label{thm4}
  Let $\ba<2-\frac{1}{p}$ and $u,v\in W^{1,p}_{loc}(\Om)$ be sub and super solution of $(P)$, respectively in the sense of definition \ref{dfn1}. Then, $u\le v$ a.e. in $\Om$.
\end{Theorem}
Next, we have a non-existence result for weak solution of $(P)$. 
\begin{Theorem}\label{thm3}
  Let $\ba\ge p$ in \eqref{eqf}. Then, there does not exist any weak solution of problem $(P)$ in the sense of definition \ref{dfn1}.
\end{Theorem}
Theorem \ref{thm3} shows that Theorem \ref{thm1} is sharp. Regarding the H\"older regularity of solutions to problem $(P)$, for the case $\de<1$, we study a more general quasilinear form of $(P)$.
Consider the following equation,
\begin{align}\label{Prb}
 -\text{div} A(x,\na u) =B(x,u,\na u)+g(x) \; \mbox{ in }\Om, \quad u=0 \; \mbox{ on }\pa\Om,
\end{align}
where $A: \ov\Om\times\mb R^N \to \mb R^N$ is a continuous function. We assume the following conditions: 
\begin{enumerate}
	\item[(A1)] $|A(x,z)|+|\pa_z A(x,z) z| \le \La (|z|^{q-1}+|z|^{p-1})\le 2\La (1+|z|^{p-1})$,
	\item[(A2)] $z.A(x,z)  \ge \nu |z|^p $,
	\item[(A3)] $\ds\sum_{i=1}^{N} |A^i(x,z)- A^i(y,z)| \le \La (1+|z|^{p-1}) |x-y|^\omega$,
	\item[(A4)] $|B(x,u,z)|\le \La (1+|z|)^p$ for $(x,u,z)\in\Om\times[-M_0,M_0]\times\mb R^N$,
\end{enumerate}
where $0<\nu\le\La$ are constants, $M_0>0$, $1<q<p<\infty$ and $\omega\in(0,1)$. Furthermore, assume $g$ satisfies the following  
 \begin{align}\label{eqg}
 	0\le g(x) \le C d(x)^{-\sg},
 \end{align}
where $\sg\in [0,1)$ and $C>0$ is a constant. First we prove H\"older continuity result up to the boundary for the gradient of the weak solution to \eqref{Prb}.  Consequently, the weak solution to $(P)$ is in $C^{1,\al}(\ov\Om)$ for the case of $\ba+\de<1$. The interior regularity follows from \cite[Theorem 1.7]{liebm91}. Furthermore, we observe that in the structure condition \eqref{lbst} of Lieberman \cite{liebm88} or that of Giacomoni et al. \cite{giacomoni}, we can not choose different $\kappa$ above and below in the inequality (which forces homogeneous type nature) contrary to our assumptions (A1) and (A2). To extend to (A1)-(A4) setting and make the article self contained as much as possible, we write a detailed proof of Lemma \ref{lem4} given crucial oscillating estimates (as in \cite[Lemma 5]{liebm88}).  Thus, taking inspiration from the works of \cite{giacomoni} and \cite{liebm88}, we consider a perturbation of the problem \eqref{Prb} (see \eqref{eq52}). We estimate various quantities involving supremum and oscillation of the gradient of the solution to the perturbed problem by means of the weak Harnack inequality, local maximum principle and suitable barrier arguments. Using these estimates, we establish control over the Campanato norm of the solution $u$ to problem $\eqref{Prb}$, which helps us to finally obtain the H\"older regularity result. Here, we would like to mention that the H\"older boundary regularity result is new for the equations involving singular nonlinearity and the gradient terms.
We state our main regularity result in this case as follows.
\begin{Theorem}\label{thm5}
 Let $u\in W^{1,p}_0(\Om)$ be a weak solution of problem \eqref{Prb} such that $0\le u\le M_0$ in $\Om$. 
 Let $\sg\in[0,1)$ in \eqref{eqg} and suppose there exists $C>0$ such that $0\le u(x)\le C d(x)$ a.e. in $\Om$. Then, there exists a constant $\al\in(0,1)$, depending only on $N,p,\omega,\sg,\nu,\La$, such that $u\in C^{1,\al}(\ov\Om)$  and
 \[ \|u\|_{C^{1,\al}(\ov\Om)} \le C(N,\nu,\La,\omega,p,\sg,M_0,\Om). \]
\end{Theorem}
Next, we prove H\"older continuity result for weak solution, $u\in W^{1,p}_{loc}(\Om)$ of problem $(P)$, for the case $\beta+\delta\ge 1$ that has not been considered so far even for the homogeneous type operators. By taking into consideration $u^\ga$, for some suitable $\ga>1$, we transform the problem $(P)$ to a new quasilinear equation  involving  a form of a weighted $(p,q)$-Laplacian operator and lower order terms (see \eqref{eq25}). Using the behaviour of $u$ near the boundary, we choose $\ga$ appropriately so that the nonlinear term in the transformed equation belongs to $L^\infty(\Om)$. Then, we follow the idea of Ladyzhenskaya and Ural'tseva \cite{ladyzh} to obtain Morrey type estimates on $u^\ga$. This proves H\"older continuity of $u^\gamma$, which in turns implies the continuity result in the sense of H\"older for $u$. The main result in this regard is given as below.
\begin{Theorem}\label{thm6}
 Let $\beta+\delta\ge 1$ and $u\in W^{1,p}_{loc}(\Om)$ be a bounded nonnegative weak solution of problem $(P)$ in the sense of definition \ref{dfn1}. Furthermore, suppose there exists $\Ga,\tl\sg>0$ such that $0\le u(x) \le \Ga d(x)^{\tl\sg}$ a.e. in $\Om$. Then, there exists $\al\in (0,1)$, depending only on $N,p,\|u\|_{L^\infty(\Om)},\tl{\sg},\Ga,\ba$ and $\de$, such that $u\in C^{0,\al}(\ov\Om)$.
\end{Theorem}
We remark that the preceding theorem complements Theorem A.1 in \cite{giacomoni2} for equations involving $p$-Laplacian with singular nonlinearity, where the solutions are considered to be in the energy space $W^{1,p}_0(\Om)$.
\begin{Corollary}\label{cor1}
 Let $u\in W^{1,p}_{loc}(\Om)$ be either the unique solution or the minimal solution (i.e., obtained as a limit of solution to the approximated problem $(P_\e)$) of problem $(P)$, then
 \begin{enumerate}
 	\item[(1)] $u\in C^{1,\al}(\ov\Om)$, for $\al\in(0,1)$ given by theorem \ref{thm5}, in the case of $\beta+\delta<1$.
 	\item[(2)] $u\in C^{0,\al}(\ov\Om)$, for $\al\in(0,1)$ given by theorem \ref{thm6}, in the case of $\beta+\delta \geq 1$.
 \end{enumerate}	
\end{Corollary}

\begin{Remark}
 We remark that our results are true for a more general class of quasilinear elliptic operators with slight modification in the proofs. Some examples of the differential operators are the following:
 \begin{enumerate}
 \item[(i)] The operator $-\De_p - a(x)\De_q$, for some non-negative function $a\in C(\ov\Om)$.
 \item[(ii)] The operator $-\textnormal{div}\big(|\na u|^{p-2}\na u + a(x) |\na u|^{q-2} \na u  \big)$, where $0\le a(x)\in W^{1,\infty}(\Om)\cap C(\ov\Om)$ with $1<q<p<\infty$.
 \item[(iii)] The operator $-\textnormal{div}\big(a(x)|\na u|^{p-2}\na u +  |\na u|^{q-2} \na u  \big)$, where $0<\inf_{\bar\Om} a(x)\le a(x)\in  C^1(\ov\Om)$ with $1<q<p$. However, the theory fails when we assume only $a(x)\ge 0$, which involves more delicate theory as its study is done in  Musielak-Sobolev spaces instead of $W^{1,p}_0(\Om)$ and the nature and behaviour of $a(x)$ dictate the ellipticity of the operator.
 \end{enumerate}
\end{Remark}
Summarizing, we observe that the leading operator in problem $(P)$ is of  non-homogeneous nature, that is, there does not exist any $\al>0$ such that  $A(t\xi)= t^\al A(\xi)$ for all $\xi\in\mb R$ and $t>0$, where $A(\xi):= -\De_p \xi -\De_q \xi$ and the nonlinear term is doubly singular in the sense that it involves two singular terms $u^{-\de}$ and the weight function which blows up near the boundary. Thus, the novelty of the paper lies in the following \\
(i) Proving behaviour of the solution near the boundary (see Lemmas \ref{lem7}, \ref{lem10} and Prop. \ref{prop2}) where the standard scaling argument fails. Whereas its homogeneous counterpart, i.e., $p=q$, uses the first eigenfunction associated with the operator $-\De_p$ with homogeneous Dirichlet boundary condition. Among other usages, this behaviour helps us to obtain Sobolev regularity and (in a forthcoming paper) to construct suitable sub and super solution to certain elliptic problems. \\
(ii) The weak comparison principle where the nonlinear term involves singularity w.r.t. the unknown function and a weight function multiplied to it, which blows up near the boundary. Here the sub and super solutions of the equation may not be in the energy space. Proofs available in the literature, even involving the homogeneous operator $-\De_p$, does not deal with the singular weight function.\\
(iii) We prove the non-existence result for weak solution of $(P)$ when the weight function grows higher than certain negative power (in this case it is $p$). This result is new even for the case $p=q$. \\
(iv) We obtain a completely new H\"older continuity for the gradient of weak solution of some quasilinear equation involving lower order terms with nonhomogeneous nature and singular nonlinearity.\\
(v) We prove H\"older continuity results for solutions of $(P)$ when $\ba+\de\ge 1$ without imposing the condition that it should lie in the energy space. The theorem in this regard generalizes the available result even for equations involving $p$-Laplacian with singular nonlinearity where the solution is required to be in the energy space. This result has many applications, for example, in proving the existence of solution to singular quasilinear systems by Schauder fixed point theorem.\par 
\noindent Turning to the layout of the paper: In Section 2, we establish our main existence theorem and optimal Sobolev regularity, here we prove Theorem \ref{thm1}. In Section 3, we prove Theorem \ref{thm4} and consequently, we obtain the uniqueness result. Here we provide proof of Theorem \ref{thm3}. In Section 4, we establish the H\"older regularity results, precisely we prove Theorems \ref{thm5} and \ref{thm6}.

\section{Existence results}
 First we consider the following perturbed problem 
   \begin{equation*}
    (P_\e) \; \left\{\begin{array}{rllll}
       -\Delta_{p}u_\e -\Delta_{q} u_\e & = f_\e(x) \big(u_\e+\e\big)^{-\delta},\; u_\e>0 \text{ in }\; \Om, \\
       u_\e&=0 \text{ on }  \pa\Om,
     \end{array}
     \right.
   \end{equation*}
 where \begin{align*}
 	f_\e(x):= \begin{cases}
 	 \big( f(x)^{\frac{-1}{\ba}}+ \e^\frac{p-1+\de}{p-\ba} \big)^{-\ba} \quad \mbox{ if }f(x)>0 \\
 	 0 \qquad\qquad \mbox{otherwise}.
 	\end{cases}
 \end{align*}
 It is easy to observe that, for $\ba<p$, the function $f_\e$ increases as $\e \downarrow 0$ and $f_\e \le f$ for all $\e>0$.
 \begin{Lemma}
  For each $\e>0$, there exists a unique solution $u_\e\in W^{1,p}_0(\Om)$ of $(P_\e)$. Furthermore, for $\ba<p$, the sequence $\{u_\e\}$ is increasing as $\e\downarrow 0$ and for each $\Om'\Subset \Om$, there exists $C_{\Om'}>0$ such that for all $\e>0$,
    \begin{align} \label{eq32}
    u_\e\ge C_{\Om'} \quad \mbox{in }\Om'. 
    \end{align}
 \end{Lemma}  
\begin{proof}
  For fixed $\e>0$ and for each $v\in L^p(\Om)$, we consider the following auxiliary problem
    \begin{equation}\label{eq31}
     \left\{\begin{array}{rllll}
     -\Delta_{p}w -\Delta_{q} w & = f_\e(x)  \big(|v|+\e\big)^{-\de},\; w>0 \text{ in }\; \Om; \ \
      w&=0 \text{ on }  \pa\Om.
     \end{array}
    \right.
   \end{equation}
 By standard minimization technique we can prove that there exists a unique solution $w\in W^{1,p}_0(\Om)$ of \eqref{eq31}. Indeed, the corresponding energy functional $J: W^{1,p}_0(\Om)\to\mb R$ defined by
  \begin{align*}
  	J(w):= \frac{1}{p} \int_{\Om}|\na w|^p dx+\frac{1}{q} \int_{\Om} |\na w|^q dx- \int_{\Om} f_\e(x) \big(|v|+\e\big)^{-\delta} w ~dx,
  \end{align*}
 is continuous, strictly convex and coercive. We define the operator $S: L^p(\Om)\to L^p(\Om)$ as follows
  \[ S(v)=w,  \]
 where $w$ is the unique solution to \eqref{eq31}. By means of  the Poincar\'e inequality, we observe that 
 \begin{align*}
 	\| S(v) \|_{L^p(\Om)}^p = \| w\|_{L^p(\Om)}^p \le C \| \na w\|_{L^p(\Om)}^p \le C \int_{\Om} (|\na w|^p+|\na w|^q ) &= C \int_{\Om} f_\e(x) \big(|v|+\e\big)^{-\delta} w \nonumber \\
 	&\le C \e^{-\de-\ba/\tau} \int_{\Om} |w| dx \nonumber \\
 	&\le C \e^{-\de-\ba/\tau}   |\Om|^\frac{p-1}{p} \| w\|_{L^p(\Om)},
 \end{align*}
 where $\tau=\frac{p-\ba}{p-1+\de}$. Then, it is standard procedure to verify that $S$ is continuous, compact and invariant on the ball of $L^p(\Om)$ with radius $\big( C \e^{-\de-\ba/\tau}  |\Om|^\frac{p-1}{p} \big)^{1/(p-1)}$. Therefore, by Schauder's fixed point theorem, there exists $u_\e\in W^{1,p}_0(\Om)$ such that $u_\e= S(u_\e)$, that is, $u_\e$ is a solution of $(P_\e)$. Since $f_\e(x) \big( |v|+\e\big)\ge 0$, by standard elliptic regularity theory, we deduce that $u_\e\ge 0$ and $u_\e\in L^\infty(\Om)$. Consequently, regularity result of Theorem \ref{thm5} with $\sg=0$ gives us $u_\e\in C^{1,\al}(\ov\Om)$ and the strong maximum principle of \cite[pp. 111, 120]{pucci} implies $u_\e>0$ in $\Om$. \\
 Next, for the case $\ba<p$, we will prove that the sequence $\{u_\e\}$ is increasing as $\e\downarrow 0$. Let $u_\e$ and $u_{\e'}$ be weak solutions of $(P_\e)$ and $(P_\e')$, respectively with $\e'\le \e$.  We observe that the term on the right in $(P_\e)$ is non-singular, therefore by density argument, we can take $(u_\e-u_{\e'})^+$ as a test function in the weak formulations. Thus, due to the fact $0\le f_\e \le f_{\e'}$, we obtain
  \begin{align*}
  	&\int_{\Om} \big( |\na u_\e|^{p-2} \na u_\e -|\na u_{\e'}|^{p-2} \na u_{\e'} \big) \na(u_\e-u_{\e'})^+ \\
  	&+ \int_{\Om} \big( |\na u_\e|^{q-2} \na u_\e -|\na u_{\e'}|^{q-2} \na u_{\e'} \big) \na(u_\e-u_{\e'})^+ \\
  	&\quad= \int_{\Om} \Big(f_\e(x)\big(u_\e+\e\big)^{-\de} - f_{\e'}(x) \big(u_{\e'}+\e'\big)^{-\de}\Big) (u_\e-u_{\e'})^+ \\ &\quad\le \int_{\Om} f_{\e'}(x) \Big(\big(u_\e+\e\big)^{-\de} - \big(u_{\e'}+\e'\big)^{-\de}\Big) (u_\e-u_{\e'})^+ \le 0. 
  \end{align*}
 Using the inequality: for $p>1$, there exists a constant $C_1=C(p)>0$  such that for all $\xi, \zeta\in\mb R^N$ with $|\xi|+|\zeta|>0$, the following holds
  \begin{align}\label{basicIn}
  	\big(  |\xi|^{p-2} \xi -|\zeta|^{p-2} \zeta \big)\cdot (\xi-\zeta) \ge C_1 \big(|\xi|+|\zeta| \big)^{p-2} |\xi-\zeta|^2,
  \end{align}
 we deduce that 
 \begin{align*}
 	\int_{\Om} \big(|\na u_\e|+|\na u_{\e'}| \big)^{p-2} |\na u_\e-\na u_{\e'}|^2 \le 0.
 \end{align*}
 This implies that $(u_\e-u_{\e'})^+=0$ a.e. in $\Om$ and therefore, $u_\e\le u_{\e'}$ in $\Om$. Consequently, \eqref{eq32} holds for all relatively compact subsets of $\Om$ on the account of $u_1\in C^{1,\al}(\ov \Om)$ and $u_1>0$ in $\Om$. For the case $\ba>p$, it is easy to see that $f_\e$ decreases as $\e\downarrow 0$ and proceeding similarly as above, we can prove that the sequence $\{u_\e\}$ is decreasing as $\e\downarrow 0$. The uniqueness of $u_\e$ follows using similar assertions and arguments used to prove monotonocity of $u_\e$ in $\e$. This completes proof of the lemma.\QED
\end{proof} 
\begin{Lemma}\label{lem7}
 Let $\ba+\de>1$ and $\ba<p$. Suppose $u_\e$ be the solution of $(P_\e)$. Then, there exist constants $\eta,\Gamma>0$, independent of $\e$, such that the following holds for $x\in\Om$,  
 \begin{align*}
  \eta \big(\big( d(x)+ \e^\frac{p-1+\de}{p-\ba} \big)^{\frac{p-\ba}{p-1+\de}} -\e\big) \le u_\e(x) \le \Gamma \big(\big( d(x)+ \e^\frac{p-1+\de}{p-\ba} \big)^{\frac{p-\ba}{p-1+\de}}-\e\big).
 \end{align*}	
 \end{Lemma}  
\begin{proof}
 Set $\tau=	\frac{p-\ba}{p-1+\de}(\in(0,1))$ and $v_\e =\eta \big( (d(x)+ \e^\frac{1}{\tau} )^\tau -\e\big)$.  Then
 \[ \na v_\e =\eta\tau \big( d(x)+ \e^\frac{1}{\tau} \big)^{\tau-1} \na d. \]
 Since $\De d\in L^\infty(\Om_\varrho)$, there exists $M>0$ such that $|\De d|\le M$ in $\Om_{\varrho}$. Therefore, for $\psi\in C_c^\infty(\Om_{\varrho})$ with $\psi \ge 0$ and noting the fact that $|\na d|=1$, we deduce that
 {\small \begin{align*}
 \int_{\Om_\varrho} -\De_p v_\e \psi &= (\eta\tau)^{p-1} \int_{\Om_\varrho} \big( d(x)+ \e^\frac{1}{\tau} \big)^{(\tau-1)(p-1)} \na d \na\psi \\
 &= (\eta\tau)^{p-1} \int_{\Om_\varrho} \left[  -\De d \; (d(x)+ \e^\frac{1}{\tau} )^{(\tau-1)(p-1)}\psi + (p-1)(1-\tau )  (d(x)+ \e^\frac{1}{\tau} )^{(\tau -1)(p-1)-1} \psi \right] \\
 &\le (\eta\tau)^{p-1} \int_{\Om_\varrho} \big( M (d(x)+ \e^\frac{1}{\tau} )^{(\tau-1)(p-1)} + (p-1) (d(x)+ \e^\frac{1}{\tau} )^{(\tau -1)(p-1)-1} \big) \psi \\
 &\le C (\eta\tau)^{p-1} \int_{\Om_\varrho} (d(x)+ \e^\frac{1}{\tau} )^{(\tau -1)(p-1)-1} \psi
 \end{align*}}
 where $C=2 \max\{M,(p-1) \}$. Similar steps yield
 {\small\begin{align*}
 \int_{\Om_\varrho} -\De_q v_\e \psi \le C (\eta\tau)^{q-1} \int_{\Om_\varrho} (d(x)+ \e^\frac{1}{\tau} )^{(\tau -1)(q-1)-1} \psi \le C (\eta\tau)^{q-1} \int_{\Om_\varrho} (d(x)+ \e^\frac{1}{\tau} )^{(\tau -1)(p-1)-1} \psi.
 \end{align*} }
 Thus, using the definition of $\tau$, we have
 \begin{align*}
 \int_{\Om_\varrho} \big(-\De_p v_\e -\De_q v_\e\big)\psi \le C \big( (\eta\tau)^{p-1} +(\eta\tau)^{q-1} \big) \int_{\Om_\varrho} (d(x)+ \e^\frac{1}{\tau} )^{-\de\tau-\ba} \psi.
 \end{align*}
Therefore, using \eqref{eqf}, we deduce that 
 \begin{align}\label{eq7}
  \frac{1}{f_\e(x)}\big(-\De_p v_\e -\De_q v_\e\big) \le C \big( \eta^{p-1} +\eta^{q-1} \big) (d(x)+ \e^\frac{1}{\tau} )^{-\de\tau} \quad \mbox{ in }\Om_{\varrho}.
 \end{align}
 Next, we observe that $(v_\e+\e)^{-\de} = \big(\eta (d+\e^{1/\tau})^\tau+(1-\eta)\e \big)^{-\de}$ and distinguish the following cases:\\
 \textit{Case (i)}: $\eta (d(x)+\e^{1/\tau})^\tau \ge (1-\eta)\e$ for $x\in\Om$.\\
 In this case, we have
 \[ \big(v_\e(x)+\e \big)^{-\de} \ge 2^{-\de}\eta^{-\de} (d(x)+\e^{1/\tau})^{-\tau\de}.  \]
 Therefore, from \eqref{eq7} for sufficiently small $\eta>0$, independent of $\e$,  we obtain
 \begin{align*}
 \frac{1}{f_\e(x)}\big(-\De_p v_\e -\De_q v_\e\big) \le C \big( \eta^{p-1} +\eta^{q-1} \big) (d(x)+ \e^\frac{1}{\tau} )^{-\de\tau} &\le 2^{-\de}\eta^{-\de} (d(x)+\e^{1/\tau})^{-\tau\de} \\
 &\le \big(v_\e(x)+\e \big)^{-\de}.
 \end{align*}
 \textit{Case (ii)}: $\eta (d(x)+\e^{1/\tau})^\tau \le (1-\eta)\e$ for $x\in\Om$.\\
 We have 
 \[ \big(v_\e(x)+\e \big)^{-\de} \ge 2^{-\de}(1-\eta)^{-\de} \e^{-\de}.  \]
 Again, we can choose $\eta>0$ small enough and independent of $\e$ so that 
 \begin{align*}
 	\frac{1}{f_\e(x)}\big(-\De_p v_\e -\De_q v_\e\big) \le C \big( \eta^{p-1} +\eta^{q-1} \big) (d(x)+ \e^\frac{1}{\tau} )^{-\de\tau}  &\le \big( \eta^{p-1} +\eta^{q-1} \big) \e^{-\de} \\ &\le 2^{-\de}(1-\eta)^{-\de} \e^{-\de} \\
 	&  \le \big(v_\e(x)+\e \big)^{-\de}.
 \end{align*}
 Therefore, in either case, we can choose $\eta>0$ sufficiently small and independent of $\e$ such that 
 \begin{align*}
 	-\De_p v_\e -\De_q v_\e \le f_\e(x) \big(v_\e(x)+\e \big)^{-\de} \quad \mbox{in }\Om_{\varrho}.
 \end{align*}
 On account of \eqref{eq32}, we choose $\eta>0$ small enough, independent of $\e$, such that in addition to the preceding relations in cases (i) and (ii), the following holds
 \begin{align*} 
 v_\e(x) \le \eta \; \text{diam}(\Om)^\tau \le C_\varrho \le u_1(x) \le u_\e(x) \quad \mbox{ in }\Om\setminus\Om_\varrho. 
 \end{align*}
 Therefore, by comparison principle, we get $v_\e \le u_\e$ in $\Om_{\varrho}$, that is,
 \begin{align*}
 \eta \big( (d+\e^{1/\tau})^\tau -\e  \big) \le u_\e(x) \quad \mbox{ in }\Om.
 \end{align*}
 Next, we will prove the upper bound for $u_\e$, for this consider 
 $w_\e =\Gamma \big( (d(x)+ \e^\frac{1}{\tau} )^\tau -\e\big)$, where $\Ga$ is a constant. Proceeding as above, for $\psi\in C_c^\infty(\Om)$ with $\psi\ge 0$, we obtain
 {\small
  \begin{align*}
  	\int_{\Om_\varrho} -\De_p w_\e \psi & 
  	= (\Gamma\tau)^{p-1} \int_{\Om_\varrho}\left[  -\De d \; (d(x)+ \e^\frac{1}{\tau} )^{(\tau-1)(p-1)}\psi + (p-1)(1-\tau ) (d(x)+ \e^\frac{1}{\tau} )^{(\tau -1)(p-1)-1} \psi \right] \\
  	&\ge (\Ga\tau)^{p-1} \int_{\Om_\varrho} \big( -M (d(x)+ \e^\frac{1}{\tau} )^{(\tau-1)(p-1)} + (p-1)(1-\tau) (d(x)+ \e^\frac{1}{\tau} )^{(\tau -1)(p-1)-1} \big) \psi.
  \end{align*} }
  On a similar note, we have
 {\small \begin{align}\label{eq62}
 \int_{\Om_\varrho} -\De_q w_\e \psi \ge (\Ga\tau)^{q-1} \int_{\Om_\varrho} \big( -M (d(x)+ \e^\frac{1}{\tau} )^{(\tau-1)(q-1)} + (q-1)(1-\tau) (d(x)+ \e^\frac{1}{\tau} )^{(\tau -1)(q-1)-1} \big) \psi. 
 \end{align}}
 \noi Furthermore, if necessary by reducing $\varrho$ further, we may assume that  there exists $C_3>0$ such that 
 \[ (p-1)(1-\tau) (d(x)+ \e^\frac{1}{\tau} )^{(\tau-1)(p-1)-1}-M (d(x)+ \e^\frac{1}{\tau} )^{(\tau-1)(p-1)} \ge C_3  (d(x)+ \e^\frac{1}{\tau} )^{(\tau -1)(p-1)-1} \ \ \mbox{in }\Om_\varrho,   \]
 and the right hand quantity in \eqref{eq62} is nonnegative (this is possible because $(\tau-1)(p-1)-1\le (\tau-1)(p-1) \le 0$). Therefore, 
 \begin{align*}
 \int_{\Om_\varrho} -\De_p w_\e \psi -\De_q w_\e \psi \ge C_3 (\Ga\tau)^{p-1}  \int_{\Om_\varrho} (d(x)+ \e^\frac{1}{\tau} )^{(\tau -1)(p-1)-1}  \psi.
 \end{align*}
Taking into account \eqref{eqf}, we obtain
\begin{align*}
 \frac{1}{f_\e(x)}\big(-\De_p w_\e -\De_q w_\e\big) \ge C_4 (\Ga\tau)^{p-1}  (d(x)+ \e^\frac{1}{\tau} )^{-\de\tau} \quad \mbox{ in }\Om_{\varrho}.
\end{align*}
By using the lower estimate of $u_\e$ by $v_\e$, for the right hand side of $(P_\e)$, we obtain
\begin{align*}
  f_\e(x) \big(u_\e+\e\big)^{-\de} \le f(x) \big(v_\e+\e\big)^{-\de} \le f(x) \eta^{-\de} d^{-\tau\de}.
\end{align*}
 Since $f\in L^\infty_{loc}(\Om)$, we observe that $f(x) \eta^{-\de} d^{-\tau\de}\in L^\infty_{loc}(\Om)$. Therefore, by $L^\infty$ estimate of \cite{ladyzh}, we get $u_\e\in L^\infty_{loc}(\Om)$ and the bound is independent of $\e$, say $\| u_\e \|_{L^\infty(\Om\setminus\Om_{\varrho})} \le K$. Now, we choose $\Ga$ sufficiently large and  independent of $\e$  satisfying last two inequalities in the following 
\begin{align*}
 w_\e=\Ga \big( (d+\e^{1/\tau})^\tau -\e  \big) \ge \Ga \big( d^\tau - \frac{\varrho^\tau}{2} \big)\ge \Ga \frac{\varrho^\tau}{2} \ge K\ge u_\e(x) \quad \mbox{ in }\Om\setminus\Om_\varrho
\end{align*}
for all $\e<\varrho^\tau/2$. Then, by comparison principle, we get $u_\e \le w_\e$ in $\Om$. 
This completes proof of the lemma.\QED
\end{proof}

\begin{Lemma}\label{lem10}
 Let $\ba+\de=1$ and $u_\e$ be the solution of $(P_\e)$. Then, there exist constants $\eta$ and $\Gamma>0$, independent of $\e$, such that the following holds in $\Om$,
  {\small\begin{align*}
  	(\eta d+ \e') \log^\frac{1}{p-\ba}\Big(\frac{L}{\eta d+\e'}\Big) -\e' \log^\frac{1}{p-\ba}\Big(\frac{L}{\e'}\Big) \le u \le (\Ga d+\e') \log^\frac{1}{p-\ba}\Big(\frac{L}{\Ga d+\e'}\Big) -\e' \log^\frac{1}{p-\ba}\Big(\frac{L}{\e'}\Big),
  \end{align*}  }
 where $L>0$ is large enough and $\e=\e' \log^{1/(p-\ba)}\big(\frac{L}{\e'}\big)$.
\end{Lemma}
\begin{proof}
 Set $\ul u_\e = (\eta d+ \e') \log^{1/(p-\ba)}\big(\frac{L}{\eta d+\e'}\big) -\e' \log^{1/(p-\ba)}\big(\frac{L}{\e'}\big)$. Then, 
  \[ \na \ul u_\e = \eta \log^\frac{1-p+\ba}{p-\ba}\Big(\frac{L}{\eta d+\e'}\Big) \left[\log\Big(\frac{L}{\eta d+\e'}\Big) -\frac{1}{p-\ba} \right] \na d. \]
 For $\psi\in C^\infty_c(\Om_\varrho)$ with $\psi\ge 0$, using the fact that $|\na d|=1$, we get 
 \begin{align*}
  \int_{\Om} -\De_p \ul u_\e \psi = \eta^{p-1} \int_{\Om} \na d \na \psi \log^\frac{(1-p+\ba)(p-1)}{p-\ba}\Big(\frac{L}{\eta d+\e'}\Big) \left[\log\Big(\frac{L}{\eta d+\e'}\Big) -\frac{1}{p-\ba} \right]^{p-1}.
 \end{align*}
A simple manipulation yields
{\small \begin{align*}
  -\De_p \ul u_\e = \frac{\eta^{p-1}}{\eta d+\e'} \Big(\log\frac{L}{\eta d+\e'}\Big)^\frac{\ba-1}{p-\ba} 
  &\left[ (-\De d) (\eta d+\e') \Big(\log \frac{L}{\eta d+\e'}\Big)^{2-p} \Big(\log\frac{L}{\eta d+\e'} -\frac{1}{p-\ba} \Big)^{p-1} \right. \\
  & \left. \ + \frac{\eta(1-p+\ba)(p-1)}{p-\ba}\Big(\log \frac{L}{\eta d+\e'}\Big)^{1-p} \Big(\log\frac{L}{\eta d+\e'} -\frac{1}{p-\ba} \Big)^{p-1} \right.\\
  & \left. \ + \eta(p-1) \Big(\log \frac{L}{\eta d+\e'}\Big)^{2-p} \Big(\log\frac{L}{\eta d+\e'} -\frac{1}{p-\ba} \Big)^{p-2}\right].
 \end{align*} }
Since $|-\De d| \le M$ in $\Om_{\varrho}$ and $\eta<1$, we deduce that 
 \begin{align*}
 	-\De_p \ul u_\e \le \frac{\eta^{p-1}}{\eta d+\e'} \Big(\log\frac{L}{\eta d+\e'}\Big)^\frac{\ba-1}{p-\ba} &\left[M (\eta d+\e') \log\frac{L}{\eta d+\e'}+ \frac{(p-1-\ba)(p-1)}{p-\ba} \right. \\
 	& \left. \quad +(p-1)\Big(1-\frac{1}{(p-\ba)\log\big({L/(\eta d+\e')}\big)} \Big)^{p-2} \right].
 \end{align*}
 Choosing $L>>1$ sufficiently large such that $\log\big(L/(diam(\Om)+1)\big)\ge 2/(p-\ba)$ and if necessary by reducing $\varrho$ further, we get $(\eta d+\e') \log\frac{L}{\eta d+\e'} \le C_1$ in $\Om_{\varrho}$. Therefore, the quantity in the bracket is bounded by a positive constant $C$, independent of $\e$. Thus, 
  \begin{align*}
  	-\De_p \ul u_\e \le C \eta^{p-1}(\eta d+\e')^{-1} \Big(\log\frac{L}{\eta d+\e'}\Big)^\frac{\ba-1}{p-\ba}.
  \end{align*}
Proceeding similarly, we obtain
 {\small \begin{align*}
  -\De_q \ul u_\e \le \frac{\eta^{q-1}}{\eta d+\e'} \Big(\log\frac{L}{\eta d+\e'}\Big)^\frac{\ba-1}{p-\ba} &\left[M (\eta d+\e') \Big(\log\frac{L}{\eta d+\e'}\Big)^\frac{q-\ba}{p-\ba} + \frac{(p-1-\ba)(q-1)}{(p-\ba)\log^\frac{p-q}{p-\ba} \big({L/(\eta d+\e')}\big)}  \right.  \\
  & \left. \ +(q-1)\Big(\log\frac{L}{\eta d+\e'}\Big)^\frac{q-p}{p-\ba} \Big(1-\frac{1}{(p-\ba)\log\big({L/(\eta d+\e')}\big)} \Big)^{q-2}  \right].
 \end{align*} }
Using the same assertions as in the estimate of $-\De_p \ul u_\e$, we get
 \begin{align*}
  -\De_q \ul u_\e \le C \eta^{q-1}(\eta d+\e')^{-1} \Big(\log\frac{L}{\eta d+\e'}\Big)^\frac{\ba-1}{p-\ba}.
 \end{align*}
 Noting the fact that $(\tl u_\e +\e)^{-\de} = (\eta d+\e')^{-\de} \Big(\log\frac{L}{\eta d+\e'}\Big)^{(\ba-1)/(p-\ba)}$ and proceeding similar to lemma \ref{lem7}, for sufficiently small $\eta>0$, independent of $\e$, we get
 \begin{align*}
  \frac{1}{f_\e(x)}\big(-\De_p \ul u_\e -\De_q \ul u_\e \big)\le  C \eta^{q-1}(\eta d+\e')^{-\de} \Big(\log\frac{L}{\eta d+\e'}\Big)^{(\ba-1)/(p-\ba)} \le \ul u_\e^{-\de} \quad\mbox{ in }\Om_{\varrho}.
 \end{align*}
Moreover, using \eqref{eq32}, we obtain 
 \[\ul u_\e(x) \le u_\e(x) \quad\mbox{ in }\Om\setminus\Om_{\varrho}, \]
 for sufficiently small $\eta>0$ independent of $\e$. Therefore, by comparison principle we deduce that $\ul u_\e \le u_\e$ in $\Om$. This gives the lower bound for $u$. To obtain the upper bound, we set 
 \[\ov u_\e = (\Ga d+ \e') \log^{1/(p-\ba)}\Big(\frac{L}{\Ga d+\e'}\Big) -\e' \log^{1/(p-\ba)}\Big(\frac{L}{\e'}\Big).\] 
 Then, proceeding as in the previous case and after simplification, we get
 {\small\begin{align*}
  -\De_p \ov u_\e \ge \frac{\Ga^{p-1}}{\Ga d+\e'} \Big(\log\frac{L}{\Ga d+\e'}\Big)^\frac{\ba-1}{p-\ba} &\left[-M (\Ga d+\e') \log\frac{L}{\Ga d+\e'} \Big(1-\frac{1}{(p-\ba)\log\frac{L}{\Ga d+\e'}}  \Big)^{p-1} \right.\\
  &\left. \quad+ \frac{\Ga(p-1-\ba)(p-1)}{(p-\ba)^2} 
    \Big(1-\frac{1}{(p-\ba)\log\frac{L}{\Ga d+\e'}\big)}  \Big)^{p-2} \frac{1}{\log\frac{L}{\Ga d+\e'}} \right. \\
  & \left. \quad +\frac{\Ga(p-1)}{p-\ba}\Big(1-\frac{1}{(p-\ba)\log\frac{L}{\Ga d+\e'}\big)} \Big)^{p-2} \right].
 \end{align*} }
 And proceeding similarly, 
{\small
 \begin{align*}
  -\De_q \ov u_\e \ge \frac{\Ga^{q-1}}{\Ga d+\e'} \Big(\log\frac{L}{\Ga d+\e'}\Big)^\frac{\ba-1}{p-\ba} &\left[-M (\Ga d+\e') \Big(\log\frac{L}{\Ga d+\e'}\Big)^\frac{q-\ba}{p-\ba} \Big(1-\frac{1}{(p-\ba)\log\frac{L}{\Ga d+\e'}}  \Big)^{q-1} \right.\\
  &\left.  + \frac{\Ga(p-1-\ba)(q-1)}{(p-\ba)^2 \Big(\log\frac{L}{\Ga d+\e'}\Big)^\frac{p-q}{p-\ba}} 
  \Big(1-\frac{1}{(p-\ba)\log\frac{L}{\Ga d+\e'}\big)}  \Big)^{q-2}  \right. \\
  & \left.  +\frac{\Ga(q-1)}{p-\ba}\Big(1-\frac{1}{(p-\ba)\log\frac{L}{\Ga d+\e'}\big)} \Big)^{q-2}\Big(\log\frac{L}{\Ga d+\e'}\Big)^\frac{p-q}{p-\ba} \right].
 \end{align*} }
We reduce $\varrho$ further so that $(\Ga d+\e')\log\frac{L}{\Ga d+\e'} \le \frac{\Ga(q-1)}{2M(p-\ba)}$  and $\log\frac{L}{\Ga d+\e'}\ge 2/(p-\ba)$ in $\Om_{\varrho}$, thus the quantity in the bracket is bounded from below by some positive constant $c$. Therefore, 
 \begin{align*}
  -\De_p \ov u_\e -\De_q \ov u_\e \ge  \frac{c\Ga^{q-1}}{\Ga d+\e'} \Big(\log\frac{L}{\Ga d+\e'}\Big)^\frac{\ba-1}{p-\ba}.
 \end{align*}
 Combining the approach of previous case with the assertions and arguments used in the case of supersolution in lemma \ref{lem7}, we obtain the required upper bound. This completes proof of the lemma. \QED
\end{proof}

\begin{Lemma}\label{lem9}
Let $\ba+\de\ge1$ and $\ba\in[0,p)$, then the sequence $\{ u_\e^{(p+\de-1)/(p-\ba)}\}$ is uniformly bounded in $W^{1,p}_0(\Om)$. Moreover, $\{u_\e\}$ is uniformly bounded in $W^{1,p}_{loc}(\Om)$.
\end{Lemma}   
\begin{proof}
 We first consider the case $\ba+\de>1$ and take $u_\e^\ga$ as a test function in the weak formulation of \eqref{eq31} for some $\ga>0$.
 Therefore, 
 \begin{align}\label{eq11}
 \int_{\Om} |\na u_\e|^{p-2} \na u_\e \na u_\e^\ga + \int_{\Om} |\na u_\e|^{q-2} \na u_\e\na u_\e^\ga =\int_{\Om} f_\e(x) \frac{u_\e^\ga}{\big( u_\e +\e\big)^{\de}} \le \int_{\Om} f(x) u_\e^{\ga-\de}.
 \end{align}
 We first observe that 
 \begin{align*}
 \int_{\Om} |\na u_\e|^{p-2} \na u_\e \na u_\e^\ga = \ga\Big(\frac{p}{p+\ga-1} \Big)^p \int_{\Om} |\na u_\e^{(p+\ga-1)/p}|^p
 \end{align*}
 and similar result holds for the second term on the left of \eqref{eq11}. Owing to \eqref{eqf} and  behaviour of $u_\e$ near the boundary proved in lemma \ref{lem7}, 
 from \eqref{eq11}, we infer that
 \begin{align*}
  \ga\Big(\frac{p}{p+\ga-1} \Big)^p \int_{\Om} |\na u_\e^{(p+\ga-1)/p}|^p \le C \int_{\Om} d(x)^{-\ba+\frac{(\ga-\de)(p-\ba)}{p+\de-1}} dx,
 \end{align*}
 the right side quantity is finite if and only if $\ga>\de+ \frac{(\ba-1)(p-1+\de)}{p-\ba}$. Thus, ${u_\e}^\rho\in W^{1,p}_0(\Om)$ is uniformly bounded for all $\rho> \frac{1}{p}\Big(p-1+\de+\frac{(\ba-1)(p-1+\de)}{p-\ba}\Big)= \frac{(p-1)(p-1+\de)}{p(p-\ba)}$. For the case $\ba+\de=1$, we take $u_\e^\de$ as a test function in the weak formulation of $(P_\e)$ and notice that the right hand side can be made independent of $u_\e$ and the function $d^{-\ba}$ is integrable, since $\ba<1$. Proceeding similarly, we obtain $\{ u_\e^{(p+\de-1)/(p-\ba)}\}$ is uniformly bounded in $W^{1,p}_0(\Om)$.\QED
\end{proof}  
Next, we will prove the existence of unique weak solution to $(P)$ when $\ba+\de<1$. To construct a suitable subsolution for this case, we recall the following proposition proved by Papageorgiou et al.\cite{papageorg}. The main ingredient of the proof is strong maximum principle of Pucci and Serrin \cite{pucci} and the strong comparison principle for general quasilinear elliptic equations.  
 For this purpose, we define the following set
 \[ \text{int }C_+:= \big\{ u\in C^1(\ov\Om): u>0 \mbox{ in }\Om, u=0 \mbox{ on }\pa\Om, \frac{\pa u}{\pa\nu} \Big|_{\pa\Om}<0 \big\}.\]
\begin{Lemma}\cite[Proposition 10]{papageorg} \label{lem6}
 For all $\rho>0$, there exists a unique solution $\tl u_\rho\in\text{int } C_+$ to the following problem
  \begin{align}\label{eqS}
  	-\De_p u -\De_q u =\rho \quad \mbox{in }\Om, \quad u=0 \ \ \mbox{on }\pa\Om.
  \end{align}
 Furthermore, the map $\rho\mapsto \tl u_\rho$ is increasing from $(0,1]$ to $C^1_0(\ov\Om)$ and $\tl u_\rho\ra 0$ in $C^1_0(\ov\Om)$ as $\rho\ra 0^+$.
\end{Lemma}
\begin{Lemma}\label{lem8}
 Let $\ba+\de<1$, then there exists a unique weak solution $u\in W^{1,p}_0(\Om)$ of $(P)$.
\end{Lemma}
\begin{proof}
 We define the energy functional $I: W^{1,p}_0(\Om)\to \mb R$ associated to  $(P)$ as follows
 \begin{align*}
 	I(u):= \frac{1}{p} \int_{\Om}|\na u|^p dx+\frac{1}{q} \int_{\Om} |\na u|^q dx- \frac{1}{1-\de}\int_{\Om} f(x) |u|^{1-\de} dx.
 \end{align*}
An easy consequence of Young inequality and Hardy inequality, for any $\varepsilon>0$, implies that
 \begin{align*}
  \frac{1}{1-\de}\int_{\Om} f(x) |u|^{1-\de} dx &\le \varepsilon \int_{\Om} \Big(\frac{|u|}{d}\Big)^p + C(\varepsilon) \int_{\Om} |u|^{\frac{p(1-\de-\ba)}{p-\ba}}\\ & \le c \varepsilon \int_{\Om} |\na u|^p + C(\varepsilon) \int_{\Om} |u|^{\frac{p(1-\de-\ba)}{p-\ba}}. 
 \end{align*}
Using the fact $p(1-\de-\ba)/(p-\ba) <p$, we infer that $I$ is coercive and weakly lower semicontinuous in $W^{1,p}_0(\Om)$. Moreover, $I$ is strictly convex on $W^{1,p}_0(\Om)_+$, the positive cone of $W^{1,p}_0(\Om)$. Therefore, there exists a unique global minimizer $u\in W^{1,p}_0(\Om)$ of $I$ and without loss of generality we may assume $u\ge 0$ a.e. in $\Om$. Now, we will prove that $u$ is in fact a solution of $(P)$. For fixed $\rho>0$, let $\tl u_\rho$ be the unique solution of \eqref{eqS} obtained in lemma \ref{lem6}. We observe that $I$ is differentiable at $\tl u_\rho$, because $\tl u_\rho\in\text{int }C_+$, and hence
 \begin{align*}
 	I^\prime(\tl u_\rho)= -\De_p \tl u_\rho -\De_q \tl u_\rho -f(x) \tl u_\rho^{-\de} = \rho -f(x) \tl u_\rho^{-\de} <0,
 \end{align*}
 for $\rho>0$ sufficiently small, since $\tl u_\rho\ra 0$ in $C^1_0(\ov\Om)$ as $\rho\ra 0^+$. Set $w= (\tl u_\rho -u)^+$ and $\xi(t)= I(u+tw)$ for $t>0$. Due to the fact $u+tw\ge t\tl u_\rho$ for $t\in(0,1]$ and Hardy's inequality, we obtain that $\xi$ is differentiable in $(0,1]$. Since $\xi$ is strictly convex, we have $t\mapsto \xi^\prime(t)$ is nonnegative and nondecreasing. Therefore, 
 \begin{align*}
  0\le \xi^\prime(1)-\xi^\prime(t) \le \xi^\prime(1) =I^\prime(\tl u_\rho)<0,
 \end{align*}
 a contradiction if support of $v$ has non zero measure. Thus, $\tl u_\rho \le u$ in $\Om$ and  since $\tl u_\rho\in\text{int }C_+$, we get $c_1 d(x)\le u$. This implies that $I$ is G\^ateaux differentiable at $u$, therefore $u$ is a weak solution of $(P)$.  \QED
\end{proof}
We now study behaviour of the solution near the boundary, for this we first prove the following proposition.
\begin{Proposition}\label{prop2}
  Let $u\in W^{1,p}_0(\Om)\cap L^\infty_{loc}(\Om)$ be a weak solution of the problem $(P)$  with $\ba+\de<1$. Then, there exists a constant $C>0$ such that 
	\begin{equation*}
	 0\le u(x)\le C \; d(x) \quad \mbox{in }\Om. 
	\end{equation*}
\end{Proposition}
\begin{proof}
 To prove the proposition, we will construct a suitable super solution to $(P)$.
 For this purpose, we recall the following observations from \cite[Lemma A.7]{giacomoni}: there exists a $C^1$ function $\Theta_\al: [0, R_\al)\to [0,\infty)$ satisfying 
	\begin{align}\label{eq4}
	 -&\frac{d}{dr} \Big( |\Theta_\al ^\prime(r)|^{p-2} \Theta_\al^\prime(r) \Big) = \Theta_\al (r)^{-\de-\ba}, \quad 0<r<R_\al \nonumber \\
	 & \Theta_\al(0)=0, \; \ \Theta_\al^\prime (0)=\al>0,
	\end{align}
 where $R_\al>0$ is the supremum of all $s\in(0,\infty)$ such that $\Theta_\al^\prime(s)>0$. We also observe that $\Theta_\al$ is strictly increasing and $\Theta_\al^\prime$ is strictly decreasing in $[0,R_\al)$. By making the substitution 
	\begin{align*}
	 \Theta_\al(r) =\al^{\frac{p}{\ba+\de-1}} \Theta_1( \al^\frac{-p}{p-1+\de+\ba} r), \; 0\le r\le R_\al, \quad R_\al = \al^\frac{-p}{p-1+\de+\ba} R_1,
	\end{align*}  
 we can choose $R_\al>0$ such that $R_\al> \text{diam}(\Om)$. Here $\Theta_1$ and $R_1$ are given by \cite[(A.19)]{giacomoni} and \cite[(A. 20)]{giacomoni}, respectively. An easy computation yields 
	\begin{align}\label{eq3} 
	 -\frac{d}{dr} \Big( |\Theta_1 ^\prime(r)|^{q-2} \Theta_1^\prime(r) \Big) \ge 0, \; \; 0<r<R_1 
	\end{align}
 and the same is true when $\Theta_1$ is replaced by $\Theta_\al$.  
 Define $w= \Ga \Theta_\al (d)$ in $\Om$, where $\Ga>1$ (to be chosen later). Then, 
	\[  \na w= \Ga \Theta_\al^\prime (d)\na d. \]
 Therefore, by observing the fact that $|\na d|=1$, for $\psi\in C_c^\infty(\Om_\varrho)$ with $\psi\ge 0$, we deduce that
	\begin{align}\label{eq5}
	 \int_{\Om_\varrho}-\De_p w \psi = \Ga^{p-1} \int_{\Om} \Theta_\al^\prime(d)^{p-1} \na d \na\psi 
	 &= \Ga^{p-1}\int_{\Om_\varrho} \Big( -\big(\Theta_\al^\prime(d)^{p-1}\big)^\prime + \Theta_\al^\prime(d)^{p-1} (-\De d) \Big) \psi \nonumber\\
	 &\ge \Ga^{p-1}\int_{\Om_\varrho} \Big( -\big(\Theta_\al^\prime(d)^{p-1}\big)^\prime - M\Theta_\al^\prime(d)^{p-1} \Big)\psi.
	\end{align}
 Similar calculation yields
	\begin{align}\label{eq5q}
	\int_{\Om_\varrho}-\De_q w \psi = \Ga^{q-1} \int_{\Om} \Theta_\al^\prime(d)^{q-1} \na d \na\psi \ge  \Ga^{q-1}\int_{\Om_\varrho} \Big( -\big(\Theta_\al^\prime(d)^{q-1}\big)^\prime - M\Theta_\al^\prime(d)^{q-1} \Big)\psi.
	\end{align}
 Therefore, coupling \eqref{eq5} and \eqref{eq5q} and using \eqref{eq4} together with \eqref{eq3}, we get
	\begin{align*}
	 -\De_p w -\De_q w &\ge \Ga^{p-1} \Big[ -\big(\Theta_\al^\prime(d)^{p-1}\big)^\prime - M\Theta_\al^\prime(d)^{p-1} \Big] +\Ga^{q-1} \Big[-\big(\Theta_\al^\prime(d)^{q-1}\big)^\prime - M\Theta_\al^\prime(d)^{q-1} \Big] \\
	 &\ge  \Ga^{p-1} \Big[ \Theta_\al(d)^{-\ba-\de} - M\Theta_\al^\prime(d)^{p-1} \Big] - \Ga^{q-1}M\Theta_\al^\prime(d)^{q-1},
	\end{align*}
 weakly in $\Om_{\varrho}$. Since $\Theta_\al$ is strictly increasing and $\Theta_\al^\prime$ is strictly decreasing together with $\Theta_\al(0)=0$ and $\Theta_\al^\prime(0)=\al$, we obtain $\Theta_\al(d)\le \al d$ and $\al^{p-1} \ge \Theta_\al^\prime(d)^{p-1}$. Therefore, if necessary, we can further reduce $\varrho>0$ such that the following holds
	\[   \Theta_\al(d)^{-\ba-\de} - M\Theta_\al^\prime(d)^{p-1} -  M\Theta_\al^\prime(d)^{q-1} \ge c \Theta_\al(d)^{-\ba-\de} \quad \mbox{ in }\Om_{\varrho}, \] 
 for some positive constant $c$. Thus,
	\begin{align*}
	-\De_p w -\De_q w \ge c \Ga^{p-1} \Theta_\al(d)^{-\ba-\de} \ge c \al^{-\ba} d^{-\ba} \Ga^{p-1} \Theta_\al(d)^{-\de}, 
	\end{align*} 
 where we used the relation $\Theta_\al(d)\le \al d$. Choosing $\Ga>0$ large enough so that $c \al^{-\ba}\Ga^{p-1}\ge c_2 \Ga^{-\de}$, we obtain 
	\begin{align*}
	-\De_p w -\De_q w \ge f(x) w^{-\de} \quad \mbox{in }\Om_{\varrho}.
	\end{align*} 
 By the fact that $u\in L^\infty_{loc}(\Om)$, we have 
  \[ \Ga \Theta_\al(d)\ge \Ga \Theta_\al(\varrho) \ge \|u\|_{L^\infty(\Om\setminus\Om_{\varrho})} \ge u(x)  \quad \mbox{in }\Om\setminus\Om_{\varrho},\] 
 for sufficiently large $\Ga$. Therefore, by comparison principle, we get 
	\begin{align*}
	u \le w= \Ga \Theta_\al(d)\le \Ga \al \;d \quad \mbox{in }\Om.
	\end{align*}
 This completes proof of the proposition. \QED
\end{proof}
\begin{Remark}
 We remark that the proof of Proposition \ref{prop2} can be used to obtain similar bounds on the bounded weak solution to the following problem, for $\de<1$,
 \begin{equation*}
  \left\{
  -\Delta_{p}u -\Delta_{q} u  = \la u^{-\delta}+u^{r-1},\; u>0 \text{ in }\; \Om; \quad
 u=0 \text{ on }  \pa\Om,
\right.
 \end{equation*}
 where $\la>0$ and $r\le p^*-1=\frac{Np}{N-p}-1$. Indeed,  since $u\in L^\infty(\Om)$, we have $u(x)\le \|u\|_\infty$ and hence 
 \begin{align*}
  -\De_p u -\De_q u =\la u^{-\de}+  u^{r-1} \le \la \big( 1+ \la^{-1} \|u\|_\infty^{r-1+\de}  \big) u^{-\de}:= \hat{\la} u^{-\de},
 \end{align*}
 where $\hat{\la}= \la\big( 1+ \la^{-1} \|u\|_\infty^{r-1+\de}  \big)$. Then, rest of the proof follows similarly, by observing the fact that if we take $u$ as a subsolution of $(P)$ instead of weak solution the proof of Proposition \ref{prop2} does not change. 
\end{Remark}
\textbf{Proof of Theorem \ref{thm1}}: For the case $\ba+\de<1$, by means of lemma \ref{lem8}, we get the existence of a weak solution $u\in W^{1,p}_0(\Om)$ satisfying $c_1 d(x)\le u(x)$ a.e. in $\Om$.  Next, since $\ba+\de<1$, following the procedure of \cite[Lemma 3.2]{dpk}, we obtain $u\in L^\infty(\Om)$. Then, applying Proposition \ref{prop2}, we obtain $u(x)\le C d(x)$ in $\Om$. \\
For the case $\ba+\de \ge 1$, due to lemma \ref{lem9}, the sequence $\{ u_\e^{(p+\de-1)/(p-\ba)}\}$ is uniformly bounded in $W^{1,p}_0(\Om)$. Therefore, we can extract a subsequence, still denoting by $u_\e$, such that $u_\e(x)\ra u(x)$ a.e. in $\Om$, for some $u\in W^{1,p}_{loc}(\Om)$. By the local H\"older regularity result of Lieberman \cite[Theorem 1.7]{liebm91}, we obtain the sequence $u_\e$ converges to $u$ in $C^1_{loc}(\Om)$. Therefore, $u$ satisfies equation $(P)$ in the sense of distribution. Moreover, from lemmas \ref{lem7} and \ref{lem10}, we deduce that 
\begin{align*}
&\eta d \log^\frac{1}{p-\ba}\Big(\frac{L}{d}\Big) \le u \le \Ga d \log^\frac{1}{p-\ba}\Big(\frac{L}{ d}\Big) \quad\mbox{ if }\ba+\de=1, \\
& \eta d(x)^{\frac{p-\ba}{p-1+\de}}  \le u(x) \le \Ga  d(x)^{\frac{p-\ba}{p-1+\de}} \quad\mbox{ if }\ba+\de>1.
\end{align*}
Repeating the proof of lemma \ref{lem9} and using above comparison estimates, we see that $u^\frac{p+\de-1}{p-\ba}\in W^{1,p}_0(\Om)$. Thus, $u$ is a weak solution to problem $(P)$ in the sense of definition \ref{dfn1}. \\
On account of the fact that the minimal weak solution (thus obtained) exhibits aforementioned behaviour near the boundary, taking $\lim_{x\ra x_0\in\pa\Om}u(x)$, we get $u\in C_0(\ov\Om)$, thus $u\in\mc C_{d_{\ba,\de}}$. For the last part of the theorem, suppose $u^\rho\in W^{1,p}_0(\Om)$, for some $\rho\ge 1$. Then from the weak formulation, it is clear that $\int_{\Om} f(x)u^{\rho-\de}<\infty$. Using behaviour of $u$ near the boundary, we see that this is equivalent to $-\ba+(\rho-\de)\frac{p-\ba}{p-1+\de}>-1$, this gives us $\rho> \frac{(p-1)(\ba+\de-1)}{p-\ba}:=\rho_0$. Furthermore, we note that $u\in W^{1,p}_0(\Om)$ if $\rho_0<1$, which yields $\de < 2+ \frac{1-\ba p}{p-1}$. This completes proof of the theorem.\QED

 \section{Comparison principle and non-existence result} 
 In this section, we first establish a comparison principle for weak sub and super solution of $(P)$ and as a consequence of this, we obtain the uniqueness result. We remark that the proof of weak comparison principle when $u,v\in W^{1,p}_0(\Om)$, is much simpler and it follows by taking $(u-v)^+$ as a test function in the weak formulation of $(P)$. \\ 
\textbf{Proof of Theorem \ref{thm4}}: For fixed $m>0$, we define $g_m: \mb R\to \mb R^+$ as follows
  \begin{align*}
 	g_m(s):= \begin{cases}
 	\min\{ s^{-\de}, m \} \quad \mbox{if }s>0 \\
 	m \qquad\mbox{ otherwise}.
 	\end{cases}
 	\end{align*}
 Let $\Upsilon_m$ be the primitive of $g_m$ such that $\Upsilon_m(1)=0$. We define a functional $\mc I_m: W^{1,p}_0(\Om)\to \mb R\cup\{-\infty,\infty\}$ as 
 	\begin{align*}
 	\mc I_m(\phi):= \frac{1}{p} \int_{\Om}|\na \phi|^p dx+\frac{1}{q} \int_{\Om} |\na \phi|^q dx- \int_{\Om} f(x) \Upsilon_m(\phi) dx,
 	\end{align*}
 for all $\phi\in W^{1,p}_0(\Om)$. Set 
 	\[ \mc M:= \{ \phi\in W^{1,p}_0(\Om): 0\le \phi\le v \mbox{ a.e. in }\Om\},\]
 which is a closed and convex set. 
 First we observe that for any bounded sequence $\{u_n\}\in\mc M$ and $\theta\in(0,1)$, to be chosen later,  
 	\begin{align*}
 	\int_{\Om} d^{-\ba} u_n dx  \le \left( \int_{\Om} \Big(\frac{u_n}{d}\Big)^{p} \right)^\frac{1-\theta}{p} \left(\int_{\Om} u_n^r\right)^\frac{\theta}{r} \left( \int_{\Om} d^{(1-\ba-\theta)l}\right)^\frac{1}{l}\le C \left(\int_{\Om} u_n^r\right)^\frac{\theta}{r} \left( \int_{\Om} d^{(1-\ba-\theta)l}\right)^\frac{1}{l},
 	\end{align*}
 where $r<p^*$, if $p<N$, $\frac{1-\theta}{p}+\frac{\theta}{r}+\frac{1}{l}=1$ and in the last inequality, we used Hardy inequality and boundedness of $\{u_n\}$ in $W^{1,p}_0(\Om)$.
 This requires $(1-\ba-\theta)l>-1$, which is equivalent to $\theta<\frac{2pr-pr\ba-r}{pr-r+p}$. 
 Due to the fact that $\ba< 2-1/p$ and by above observation, it is easy to deduce that $\mc I_m$ is weakly lower semicontinuous on $\mc M$. Therefore, there exists a minimizer $w$ of $\mc I_m$ in $\mc M$ and the following holds
 	\begin{align}\label{eq44}
 	 \int_{\Om} \big(|\na w|^{p-2} + |\na w|^{q-2} \big) \na w \na(\phi-w) dx\ge \int_{\Om} f(x) \Upsilon_m^\prime(w) (\phi-w) dx 
 	\end{align}
 for $\phi\in w+ \big(W^{1,p}_0(\Om)\cap L^\infty_c(\Om) \big)$ with $0\le \phi\le v$ a.e. in $\Om$. \\ 
 \textbf{Step I}: We claim that, for all $\phi\in C_c^\infty(\Om)$ with $\phi\ge 0$, there holds
 	\begin{align}\label{eq45}
 	 \int_{\Om} \big(|\na w|^{p-2} + |\na w|^{q-2} \big) \na w \na\phi dx\ge \int_{\Om} f(x) \Upsilon_m^\prime(w)\phi~ dx.
 	\end{align} 
 Let $h\in C_c^\infty(\mb R)$ such that $0\le h \le 1$, $h\equiv 1$ in $[-1,1]$ and $\text{supp}(h)\subset (-2,2)$. Now, for $\phi\in C_c^\infty(\Om)$ satisfying $\phi\ge 0$ in $\Om$, we define $\phi_k:= h(\frac{w}{k})\phi$ and $\phi_{k,t}:= \min\{w+t\phi_k, v\}$, for $k\ge 1$ and $t>0$. It is easy to observe that $\phi_{k,t}\in w+\big(W^{1,p}_0(\Om)\cap L^\infty_c(\Om) \big)$ with $0\le \phi_{k,t}\le v$ a.e. in $\Om$. From \eqref{eq44}, we infer that
 	\begin{align*}
 	 \int_{\Om} \big(|\na w|^{p-2} + |\na w|^{q-2} \big) \na w \na(\phi_{k,t}-w) dx\ge \int_{\Om} f(x) \Upsilon_m^\prime(w) (\phi_{k,t}-w) dx.
 	\end{align*}
 Now, using the inequality \eqref{basicIn}, we deduce that
  \begin{align*}
 	 c\int_{\Om} \big(|\na w|+|\na\phi_{k,t}|\big)^{p-2}|\na(\phi_{k,t}-w)|^2 &\le \int_{\Om} \big(|\na\phi_{k,t}|^{p-2} \na\phi_{k,t} - |\na w|^{p-2}\na w  \big) \na(\phi_{k,t}-w) dx \\
 	 &\ \ +\int_{\Om} \big(|\na\phi_{k,t}|^{q-2} \na\phi_{k,t} - |\na w|^{q-2}\na w  \big) \na(\phi_{k,t}-w) dx \\
 	 &\le \int_{\Om} \big(|\na \phi_{k,t}|^{p-2} + |\na \phi_{k,t}|^{q-2} \big) \na \phi_{k,t} \na(\phi_{k,t}-w) \\
 	 & \quad- \int_{\Om} f(x) \Upsilon_m^\prime(w) (\phi_{k,t}-w) dx.
  \end{align*}
 This implies that 
  \begin{align*}
 	c&\int_{\Om}\big(|\na w|+|\na\phi_{k,t}|\big)^{p-2}|\na(\phi_{k,t}-w)|^2 -\int_{\Om} f(x) \big(\Upsilon_m^\prime(\phi_{k,t})-\Upsilon_m^\prime(w)\big) (\phi_{k,t}-w) \\
 	& \ \le \int_{\Om} \big(|\na \phi_{k,t}|^{p-2} + |\na \phi_{k,t}|^{q-2} \big) \na \phi_{k,t} \na(\phi_{k,t}-w-t\phi_k)- \int_{\Om} f(x) \Upsilon_m^\prime(\phi_{k,t}) (\phi_{k,t}-w-t\phi_k) \\
 	& \quad+t\Big[\int_{\Om} \big(|\na \phi_{k,t}|^{p-2} + |\na \phi_{k,t}|^{q-2} \big) \na \phi_{k,t} \na\phi_k - \int_{\Om} f(x) \Upsilon_m^\prime(\phi_{k,t}) \phi_k\Big].
  \end{align*}
 Simplifying it further and using the observation that the first term on the left is nonnegative, we obtain
 	\begin{equation}\label{eq46}
 	\begin{aligned}
 	 -\int_{\Om} f(x)  \big(\Upsilon_m^\prime(\phi_{k,t})-\Upsilon_m^\prime(w)\big)  (\phi_{k,t}-w) &\le \int_{\Om} \big(|\na v|^{p-2} + |\na v|^{q-2} \big) \na v \na(v-w-t\phi_k) \\
 	 & \ \ - \int_{\Om} f(x) \Upsilon_m^\prime(\phi_{k,t}) (\phi_{k,t}-w-t\phi_k)\\
 	 & \ \ +t\Big[\int_{\Om} \big(|\na \phi_{k,t}|^{p-2} + |\na \phi_{k,t}|^{q-2} \big) \na \phi_{k,t} \na\phi_k \\
 	 &\qquad\quad - \int_{\Om} f(x) \Upsilon_m^\prime(\phi_{k,t}) \phi_k\Big].
 	 \end{aligned}
 	\end{equation}
 From the definition of $\Upsilon_m$, it is clear that $v$ is super solution to the following equation
 	\begin{align*}
 	-\De_p v -\De_q v = f(x)\Upsilon_m^\prime(v).
 	\end{align*}
 Therefore, from \eqref{eq46}, we obtain
 	\begin{align*}
 	-\int_{\Om} f(x) \big(\Upsilon_m^\prime(\phi_{k,t})-\Upsilon_m^\prime(w)\big) (\phi_{k,t}-w) &\le t\Big[\int_{\Om} \big(|\na \phi_{k,t}|^{p-2} + |\na \phi_{k,t}|^{q-2} \big) \na \phi_{k,t} \na\phi_k  \\
 	&\qquad\quad -\int_{\Om} f(x) \Upsilon_m^\prime(\phi_{k,t}) \phi_k\Big].
 	\end{align*}
 Since supports of $\phi_{k,t}-w$ and $\phi_k$ are compact, using dominated convergence theorem, we pass the limit $t\ra 0$. Thus, 
 	\begin{align*}
 	\int_{\Om} \big(|\na w|^{p-2} + |\na w|^{q-2} \big) \na w \na\phi_k -\int_{\Om} f(x) \Upsilon_m^\prime(w) \phi_k \ge 0.
 	\end{align*}
 Taking $k\ra \infty$, we complete the proof of \eqref{eq45}. \\
 \textbf{Step II}: In this step we will show that $u \le w+\e$ in $\Om$ for all $\e>0$.\\
 Since $w\in W^{1,p}_0(\Om)$, the function $(u-w-\e)^+$ is in $W^{1,p}_0(\Om)$. By density argument and Fatou lemma, we see that \eqref{eq45} holds if $T_k\big((u-w-\e)^+\big)$ is taken as a test function, that is,
 	\begin{align}\label{eq47}
 	\int_{\Om} \big(|\na w|^{p-2} + |\na w|^{q-2} \big) \na w \na T_k\big((u-w-\e)^+\big) \ge \int_{\Om} f(x) \Upsilon_m^\prime(w) T_k\big((u-w-\e)^+\big),
 	\end{align}
 where $T_k(s) = \min\{s,k\}$. Let $\tl\phi_n\in C_c^\infty(\Om)$ be such that $\tl\phi_n \ra (u-w-\e)^+$ in $W^{1,p}_0(\Om)$. Set $\phi_{n,k}:= T_k\big( \min\{(u-w-\e)^+,\tl\phi_n^+ \}\big)$. It is easy to observe that $\phi_{n,k}\in W^{1,p}_0(\Om)\cap L^\infty_c(\Om)$, therefore by density argument, we obtain
 	\begin{align*}
 	\int_{\Om} \big(|\na u|^{p-2}+|\na u|^{q-2} \big) \na u \na \phi_{n,k} \le \int_{\Om} f(x)u^{-\de} \phi_{n,k}.
 	\end{align*}
 Consequently, using dominated convergence theorem, we get
  \begin{align}\label{eq48}
 	\int_{\Om} \big(|\na u|^{p-2}+|\na u|^{q-2} \big) \na u \na T_k\big((u-w-\e)^+\big) \le \int_{\Om} f(x)u^{-\de} T_k\big((u-w-\e)^+\big).
  \end{align}
 For $m>\e^{-\de}$, proceeding similar to lemma \ref{lem9}, subtracting \eqref{eq48} from \eqref{eq47}, we deduce that 
  \begin{align*}
 	c \int_{\Om} \big( |\na u| +|\na w| \big)^{p-2} |\na T_k\big((u-w-\e)^+\big)|^2 &\le \int_{\Om} f(x) \big( u^{-\de} - \Upsilon_m^\prime(w)\big) T_k\big((u-w-\e)^+\big) \\
 	&\le \int_{\Om} f(x) \big( \Upsilon_m^\prime(u)- \Upsilon_m^\prime(w)\big) T_k\big((u-w-\e)^+\big) \\
 	&\le 0,
  \end{align*}
 where the last inequality holds in the support of $(u-w-\e)^+$. This implies that $T_k\big((u-w-\e)^+\big)= 0$ a.e. in $\Om$ and since it is true for every $k>0$, we get $u \le w+\e$ in $\Om$. By the arbitrariness of $\e$ and the fact that $w\le v$, we obtain the required result of the theorem. \QED
 \begin{Corollary}
 	Let $\ba<2-1/p$, then there exists a unique weak solution of $(P)$.
 \end{Corollary}   
\begin{proof}
 Suppose there exist two weak solutions $u$ and $v$ of problem $(P)$ in $W^{1,p}_{loc}(\Om)$. Then, we can treat $u$ as a subsolution and $v$ as a supersolution to $(P)$. Consequently, the comparison principle implies $u\le v$  a.e. in $\Om$. Reversing the role, we get $u=v$ a.e. in $\Om$. \QED
\end{proof}

\textbf{Proof of Theorem \ref{thm3}}: On the contrary suppose there exists a solution $u_0\in W^{1,p}_{loc}(\Om)$ of $(P)$ and $\ga_0\ge 1$ such that $u_0^{\ga_0}\in W^{1,p}_0(\Om)$.
From \eqref{eqf}, we have
\begin{align*}
c_1 \; d(x)^{-\ba} \le f(x) \le c_2 \; d(x)^{-\ba} \quad \mbox{in }\Om_\varrho.
\end{align*}
For $\tl\ba<p$, which we will specify later, we choose $f_{\tl\ba}\in L^\infty_{loc}(\Om)$ such that $m f_{\tl\ba}(x)\le f(x)$ a.e. in $\Om$, for some constant $m\in(0,1)$ independent of $\tl\ba$, and for some positive constants $c_3,c_4$,
\begin{align*}
c_3 \; d(x)^{-\tl\ba} \le m f_{\tl\ba}(x) \le c_4 \; d(x)^{-\tl\ba} \quad \mbox{in }\Om_\varrho.
\end{align*}
Now we will construct a suitable subsolution near the boundary $\pa\Om$, to arrive at some contradiction. For $\e>0$, let $w_\e\in W^{1,p}_0(\Om)$ be the unique solution to the following problem
\begin{align}\label{eq40}
-\Delta_{p}w_\e -\Delta_{q} w_\e = mf_{\tl\ba,\e}(x) \; (w_\e+\e)^{-\delta}, 
\end{align}
where $ f_{\tl\ba,\e}(x):= 
\big( f_{\tl\ba}(x)^{\frac{-1}{\tl\ba}}+ \e^\frac{p-1+\de}{p-\tl\ba} \big)^{-\tl\ba}$ $ \mbox{ if }f_{\tl\ba}(x)>0$ and  
$0$ otherwise.

Next, we will prove that $w_\e \le u_0$ in $\Om$. We observe that $w_\e\in C^{1,\al}(\ov\Om)$ and $w_\e=0$ on $\pa\Om$. Therefore, for given $\sg>0$, there exists $\rho>0$ such that $w_\e \le \sg/2$ in $\Om_\rho$. Moreover, $w_\e -u_0-\sg\le -\sg/2<0$ in $\Om_\rho$, because $u_0\ge 0$, we have 
\[ \text{supp}(w_\e-u_0-\sg)^+ \subset \Om\setminus\Om_\rho\Subset\Om. \]
Therefore, $(w_\e-u_0-\sg)^+\in W^{1,p}_0(\Om)$ and from the weak formulation of \eqref{eq40}, we obtain
\begin{align}\label{eq41}
\int_{\Om} \big(|\na w_\e|^{p-2} + |\na w_\e|^{q-2} \big)\na w_\e \na T_k\big((w_\e-u_0-\sg)^+\big) = \int_{\Om} \frac{mf_{\tl\ba,\e}(x)}{ (w_\e+\e)^{\de}} T_k\big((w_\e-u_0-\sg)^+\big),
\end{align}
where $T_k(s):= \min\{ s, k\}$ for $k>0$ and $s\ge 0$. Furthermore, since $u_0\in W^{1,p}_{loc}(\Om)$ is a weak solution to $(P)$, for all $\psi\in C_c^\infty(\Om)$, we have 
\begin{align}\label{eq42}
\int_{\Om} \big(|\na u_0|^{p-2}+|\na u_0|^{q-2} \big) \na u_0 \na \psi = \int_{\Om} f(x)u_0^{-\de}\psi.
\end{align}
Let $\psi_n\in C_c^\infty(\Om)$ be such that $\psi_n \ra (w_\e-u_0-\sg)^+$ in $W^{1,p}_0(\Om)$. Set $\tl \psi_{n,k}:= T_k\big( \min\{(w_\e-u_0-\sg)^+,\psi_n^+ \}\big)$. Then, $\tl \psi_{n,k}\in W^{1,p}_0(\Om)\cap L^\infty_c(\Om)$, therefore from \eqref{eq42}, we infer that 
\begin{align*}
\int_{\Om} \big(|\na u_0|^{p-2}+|\na u_0|^{q-2} \big) \na u_0 \na \tl\psi_{n,k} = \int_{\Om} f(x)u_0^{-\de} \tl\psi_{n,k}.
\end{align*}
Using the fact that $\text{supp}(w_\e-u_0-\sg)^+\Subset \Om$ and Fatou lemma, we obtain
\begin{align}\label{eq43}
\int_{\Om} \big(|\na u_0|^{p-2}+|\na u_0|^{q-2} \big) \na u_0 \na T_k\big((w_\e-u_0-\sg)^+\big) &\ge \int_{\Om} f(x)u_0^{-\de} T_k\big((w_\e-u_0-\sg)^+\big) \nonumber\\
&\ge \int_{\Om} mf_{\tl\ba,\e}(x) u_0^{-\de} T_k\big((w_\e-u_0-\sg)^+\big).
\end{align}
Taking into account \eqref{eq41} and \eqref{eq43}, we deduce that 
\begin{align*}
&\int_{\Om} \big( |\na w_\e|^{p-2} \na w_\e -|\na u_0|^{p-2}\na u_0 \big)\na T_k\big((w_\e-u_0-\sg)^+\big) dx \\
&\qquad+ \int_{\Om} \big( |\na w_\e|^{p-2} \na w_\e -|\na u_0|^{p-2}\na u_0 \big)\na T_k\big((w_\e-u_0-\sg)^+\big) dx \\
&\le \int_{\Om} mf_{\tl\ba,\e}(x) \big((w_\e+\e)^{-\de} -u_0^{-\de} \big) T_k\big((w_\e-u_0-\sg)^+\big)dx \\
&\le \int_{\Om} m f_{\tl\ba,\e}(x) \big(w_\e^{-\de} -u_0^{-\de} \big) T_k\big((w_\e-u_0-\sg)^+\big) dx \le 0.
\end{align*}
To estimate the quantities on the left side, we use the inequality \eqref{basicIn}, therefore
\begin{align*}
C\int_{\Om} \big( |\na w_\e| +|\na u_0| \big)^{p-2} |\na T_k\big((w_\e-u_0-\sg)^+\big)|^2 \le 0,
\end{align*}
this implies that $T_k\big((w_\e-u_0-\sg)^+\big)= 0$ a.e. in $\Om$ and since it is true for every $k>0$, we get $w_\e \le u_0+\sg$ in $\Om$. Moreover, the arbitrariness of $\sg$ proves $w_\e\le u_0$ in $\Om$.
Owing to the estimates of $w_\e$ given by lemma \ref{lem7}, we have 
\[ \eta \big(\big(d(x)+ \e^\frac{p-1+\de}{p-\tl\ba} \big)^{\frac{p-\tl\ba}{p-1+\de}}-\e \big) \le w_\e(x) \le u_0(x) \quad \mbox{in } \Om.\] 
Since $u_0^{\ga_0}\in W^{1,p}_0(\Om)$, by Hardy inequality, we obtain 
\begin{align*}
\eta^{\ga_0p} \int_{\Om} \frac{\big(\big(d(x)+ \e^\frac{p-1+\de}{p-\tl\ba} \big)^{\frac{p-\tl\ba}{p-1+\de}}-\e \big)^{\ga_0 p}}{d^p} \le C \int_{\Om} \frac{u_0^{\ga_0 p}}{d^p} <\infty,
\end{align*}
choosing $\tl\ba<p$ sufficiently close to $p$ and taking $\e\downarrow 0$, we obtain the quantity on the left side is not finite, which yields a contraction. This completes proof of the theorem. \QED

 \section{H\"older regularity}
 In this section, we prove H\"older regularity results for gradient of weak solution to equation \eqref{Prb}, which is a general form of $(P)$. Subsequently, we prove H\"older continuity of weak solution to problem $(P)$ for the case $\ba+\de\ge 1$. 
 The interior H\"older regularity for solutions of \eqref{Prb} follows from \cite[Theorem 1.7, p.320]{liebm91}. To prove regularity results up to the boundary for gradient in the first case,
we flatten the boundary $\pa\Om$ by a $C^2$ diffeomorphism $\Phi$ such that 
$d(\Phi(x)) =\big(\Phi(x)\big)_N$ for all $x\in\mb R^N$, where $d$ denotes the distance from an open ball centered at the origin with $\Phi(x)_N\ge 0$. For $r>0$, we fix the following notation
{\small\begin{align*}
  B^+_r(y) &= \{ x\in\mb R^N : |x-y|<r, x_N-y_N>0 \}, \; B^0_r(y)= \{ x\in\mb R^N : |x-y|<r, x_N-y_N=0 \}. 
\end{align*}}
 As a consequence of the above transformation, problem in \eqref{Prb} takes the following form 
 \begin{align}\label{eq51}
 	-\text{div} A(x,\na u) =B(x,u,\na u)+g(x) \; \mbox{ in } B^+_1(0), \quad u=0 \; \mbox{ on }B_1^0(0).
 \end{align}
Conditions in Theorem \ref{thm5} imply 
 \begin{enumerate}
 	\item[(B1)] $0\le g(x)\le C x_N^{-\sg}$ for a.e. $B_1^+(0)$.
 	\item[(B2)] $0\le u(x)\le C x_N$ for a.e. $B_1^+(0)$.
 \end{enumerate} 
 For any $x_0\in B^+_{1/2}(0)$ and $0<R<1/2$, we consider the following perturbed problem
 \begin{align}\label{eq52}
 	-\text{div} A(x_0,\na v) =0 \; \mbox{ in } B^+_R(x_0), \quad v=u \; \mbox{ on }\pa B_R^+(x_0).
 \end{align}
 To estimate the various quantities involving the solution $v$, we consider the normalized form of problem \eqref{eq52} with $x_0=0\in \mb R^N$, that is,
 \begin{align}\label{eq53}
 -\text{div} A(0,\na v) =0 \; \mbox{ in } B^+_R(0), \quad v=u \; \mbox{ on }\pa B_R^+(0).
 \end{align}
 In what follows, we denote $B^+_r(0)$ as $B^+_r$ for all $r>0$. The next lemma is inspired from \cite[Lemma 5]{liebm88}.
\begin{Lemma}\label{lem4}
  There exists a unique solution $v\in W^{1,p}(B_R^+)$ of \eqref{eq53}. Furthermore, the following hold
	\begin{enumerate}
	  \item[(i)] $\sup_{B^+_{R/2}} |\na v| \le C \Big(R^{-\frac{N}{p}} \| \na v \|_{L^p(B_R^+)}+\chi\Big)$,
	  \item[(ii)] $\textnormal{osc}_{B_r^+} \na v \le C \big(\frac{r}{R}\big)^\vsg \Big(\sup_{B^+_{R/2}} |\na v|+\chi \Big)$ \quad \mbox{for } $0<r<R/7$,
	  \item[(iii)] $\int_{B^+_R} |\na v|^p \le C \int_{B^+_R} \big( 1+|\na u|\big)^p dx$,
	  \item [(iv)] $\sup_{B^+_{R}} |u-v| \le \textnormal{osc}_{B^+_{R}}  u \le \sup_{B^+_{R}} u\le C R$.
	 \end{enumerate}
  Here $\chi>0$ is a constant which depends only on $\La,\nu$ and $p$, and the constants $C,\vsg>0$ depend only on $\La,\nu,p,N$ and $\omega$. 
\end{Lemma} 
\begin{proof}
 Proof of existence of the unique solution to \eqref{eq53} is standard.
 Then, we first prove (i) and (iii).
 To prove (i), we will apply the local maximum principle proved by Trudinger in \cite{trudinger} for the case of cubes, however the same proof can be adopted to get similar results for the case of balls, see for instance remark at the end of \cite[Theorem 1.2, p. 316]{liebm91}. Applying the aforementioned result to the nonnegative weak solution $v$ of \eqref{eq53}, for $\rho\in(0,1)$, we have
   \begin{align*}
   	\sup_{B^+_{\rho R}} v \le \frac{C}{(1-\rho)^{N}} \Big( R^{-\frac{N}{p}} \| v\|_{L^p(B^+_R)} + \chi R \Big),
   \end{align*}
 where $C,\chi>0$ are constants depending only on $N, \La, \nu, p, \|v\|_{L^\infty(B_R^+)}$. By means of the Poincar\'e inequality (note that $v=0$ on $B_R^0$), the above equation implies that
  \begin{align}\label{eq12}
  	\sup_{B^+_{\rho R}} \frac{v}{R} \le \frac{C}{(1-\rho)^{N}} \left( R^{-\frac{N}{p}} \Big(\int_{B_R^+} | \na v|^p~ dx \Big)^{1/p} + \chi \right).
  \end{align}
 To estimate the term on the left, we will use barrier argument as in \cite[Section 2]{liebm88}. Let us fix $r\in (0,R)$, $x_0\in B^+_{r/4}$ and set 
    \begin{align*}
      w(x)= 16 r^{-2} \sup_{B^+_{r}} v \; \Big( |x-(x_0',0)|^2 + N\frac{\La}{\nu} (rx_N - x_N^2) \Big) \quad \mbox{ for }x\in B_r^+.
    \end{align*}  
  It is easy to observe that $w\ge v$ on $\pa B_{r/2}^+$ and by direct computation, we have
   \begin{align*}
   	 -\text{div} A(0,\na w) \ge 0 \; \mbox{ in } B^+_{r/2}.
   \end{align*}
 That is, $w$ is a classical super solution, therefore employing the maximum principle, we obtain $w\ge v$ in $B_{r/2}^+$. Evaluating these at $x=x_0$, 
    \begin{align*}
     \frac{v(x_0)}{x_{0N}}\le C r^{-1} \sup_{B^+_{r}} v,
    \end{align*}
 thus, by arbitrariness of $x_0$, we get
  \begin{align}\label{eq22}
  	\sup_{B^+_{r/4}} \frac{v(x)}{x_N} \le C\sup_{B^+_{r}} \frac{v(x)}{r}.
  \end{align}
Before proceeding further, we recall the following Caccioppoli type inequality: for any $\varrho>0$ and $x_1\in\mb R^N$ such that $B_\varrho(x_1)\subset B_R^+$, the following holds
\begin{align}\label{cacciopp}
 \int_{B_{\varrho/2}} |\na v|^p dx\leq C \int_{B_{\varrho}}\left(1+\Big(\frac{v}{\varrho}\Big)^p  \right)dx,
\end{align}
where $C$ is a positive constant (independent of $v$). Indeed, by taking $\xi^p v$ as a test function in the weak formulation of \eqref{eq53}, where $\xi\in C_c^\infty(B_R^+)$ is a cut off function, and using the structure conditions of $A$ (A1-A2) together with Young's inequality, we get
 \begin{align*}
 	\int_{B^+_{R}} \xi^p|\na v|^p dx \leq C \int_{B^+_R}\left(\xi^p+|v|^p|\na\xi|^p  \right)dx.
 \end{align*}
Specifying the choice of $\xi$ in the above expression, we obtain \eqref{cacciopp}. 
 Next, we estimate $\sup|\na v|$ by virtue of \eqref{eq22} in $B_{r/16}^+$. Let $x_0\in B_{r/16}^+\setminus B^0_R$ and $x_1$ be the projection of $x_0$ on $B_R^0$. Then, $d:=d(x_0, B_R^0)\le r$. By the interior gradient estimate of \cite[(5.3a)]{liebm91} and Caccioppoli's inequality \eqref{cacciopp}, we deduce that
 \begin{align*}
 	|\na v(x_0)|^p \le C d^{-N} \int_{B_\frac{d}{4}(x_0)} (1+|\na v|^p) dx\le C d^{-N} \int_{B_\frac{d}{2}(x_0)} \Big(1+\big(\frac{v}{d}\big)^p\Big) \le C \sup_{B_\frac{d}{2}(x_0)} \Big(1+\big(\frac{v}{d}\big)^p\Big).
 \end{align*}
 Since $x_N/2 \le d$ in $B_{d/2}(x_0)$ and $B_{d/2}(x_0)\subset B^+_{2d}(x_1)\subset B^+_{r/8}(x_1)\subset B^+_{r/4}$, from above inequality and \eqref{eq22}, we obtain
 \begin{align*}
 	|\na v(x_0)|^p \le C \sup_{B_{d/2}(x_0)} \Big(1+\big(\frac{v}{d}\big)^p\Big) \le C \sup_{B^+_{r/4}(0)} \Big(1+\big(\frac{v}{x_N}\big)^p\Big) \le C \sup_{B^+_{r}} \Big(1+\big(\frac{v}{r}\big)^p\Big).
 \end{align*}
 Using the fact that $p>1$ and arbitrariness of $x_0\in B^+_{r/16}$, we get 
 \begin{align*}
 	\sup_{B^+_{r/16}}|\na v| \le C \Big(1+\sup_{B^+_{r}} \frac{v}{r}\Big).
 \end{align*}
On account of interior gradient estimate \cite[(5.3a)]{liebm91} and Caccioppoli inequality in the interior, we have similar bound for interior balls too. Therefore, by covering argument, for suitable $r>0$, we obtain 
 \begin{align}\label{eq13}
 	\sup_{B^+_{R/2}}|\na v| \le C \Big(1+\sup_{B^+_{R}} \frac{v}{R}\Big).
 \end{align}
 Now, coupling \eqref{eq12} and \eqref{eq13}, we get the required result in (i).  Proof of (iii) follows exactly on the same lines of proof of \cite[Lemma 5, (3.3)]{liebm88}. \QED
\end{proof} 
To complete the proof of lemma \ref{lem4}, we need the following two lemmas.
 \begin{Lemma}\label{lem2}
  Let $L$ be an elliptic operator of the form $Lu = a^{ij} D_{ij} u$ with 
 \begin{align*}
  \nu |z|^{p-2}|\xi|^2 \le a^{ij}(z) \xi_i \xi_j \le \La (|z|^{p-2}+|z|^{q-2})|\xi|^2 \quad \mbox{for }x\in B_1^+, \ z,\xi\in\mb R^N,
 \end{align*} 
 where $\La,\nu$ are positive constants with $\nu\le\La$. Furthermore, let us assume $u\in C^2(B_1^+)$ be such that $0\le u \le H x_N$ in $B_1^+$ and $L u(x) = 0$. Then, for $\rho=\rho(N,\La,\nu)$, small enough and $R< 1$, there exist positive constants $C$ and $\vsg$, depending only on $N,\La,\nu$, such that the following holds
    \begin{align*}
      \textnormal{osc}_{B_r^+} \frac{u}{x_N} \le C \Big(\frac{r}{R} \Big)^\vsg \left(\textnormal{osc}_{B_{R/2}^+} \frac{u}{x_N} + \chi \right),
    \end{align*}
  for $r\in(0,R/2)$.
 \end{Lemma}
 \begin{proof}
 Proof follows using the weak Harnack inequality of \cite{liebm91} and can be completed proceeding similarly to \cite[Lemma 5.2]{liebm86}. 
\QED
 \end{proof}
Next, we have the control over oscillation of $\na v$, which is inspired from \cite[Lemma 4]{liebm88}.
\begin{Lemma}\label{lem5}
  Let $v\in W^{1,p}(B_R^+)$ be the unique solution of \eqref{eq53} and assume $v\in C^2(B_R^+)$. Then, there exists a positive constant $C=C(N,\omega,\La,p)$ such that
    \begin{align*}
       \textnormal{osc}_{B^+_r} \na v \le C \Big(\frac{r}{R} \Big)^\vsg \Big(\sup_{B^+_{R/2}} |\na v| + \chi \Big),
    \end{align*}
  for $r\in(0,R/7)$.
\end{Lemma} 
\begin{proof}
  As a consequence of Lemma \ref{lem2}, we can define
  \begin{align*}
  	w(x)=v(x)-\pa_{x_N} v(0) x_N.
  \end{align*}
Therefore, for any $y\in B^+_{r}$, by the interior regularity estimate of \cite[Theorem 1.7]{liebm91} and using equivalence form of Campanato and H\"older norms, we have 
  \begin{align}\label{eq23}
  	\text{osc}_{B_\rho(y)} \na w \le C \Big(\frac{\rho}{r} \Big)^\al \big(\text{osc}_{B_r(y)} \na w + \chi r^\sg \big),
  \end{align}
 for $\sg>0$, if $0<\rho<r$.
Next, let $D=B_{\frac{y_N}{2}}(y)$ and $\e\in(0,y_N)$. For $x,y\in D_\e:= \{ x\in D: \text{dist}(x,  B^0_1)>\e \}$, if $|x-y|<\e$, we take $\rho=|x-y|$ and $r=\e$ in \eqref{eq23}, thus
 \begin{align*}
 	\e^\al |\na w(x)-\na w(y)| |x-y|^{-\al} \le C \big(\sup_{D_\e} |\na w|+\chi \e^\sg\big).
 \end{align*}
 The above inequality is trivially true if $|x-y|\ge \e$. Therefore, multiplying by $\e$ and taking supremum over $\e\in(0,y_N)$, we obtain
  \begin{align}\label{eq24}
    [w]^*_{\al+1, D}:=\sup_{\e>0}\left(\e^{\al+1} \sup_{D_\e}\frac{|\na w(x)-\na w(y)|}{|x-y|^\al} \right) \le C \Big(\sup_{\e>0} \big(\e\sup_{D_\e} |\na w|\big) + \chi y_N^{\sg+1} \Big).  	
  \end{align}
  Using the standard interpolation identity
  \[ \sup_{\e>0} \big(\e\sup_{D_\e} |\na w|\big) \le 2 \mu^{-1} \|w\|_{L^\infty(D)}+ 2^{1+\al} \mu^\al  [w]^*_{\al+1, D} \quad \forall \ \mu\in(0, 1/2],  \]
  for suitable $\mu$ and \eqref{eq24}, we obtain
  \begin{align}\label{eq15}
  	y_N |\na w(y)|\le \sup_{\e>0} \big(\e\sup_{D_\e} |\na w|\big) \le C \Big( \sup_{B_{\frac{y_N}{2}}(y)} |w| +\chi y_N^{\sg+1} \Big).
  \end{align}
 For $x\in B_{\frac{y_N}{2}}(y)$, we have $y_N/2 \le x_N \le 3y_N/2$ and since $y\in B^+_r$, we get $B_{\frac{y_N}{2}}(y) \subset B^+_{3r/2}$. Then, for $x\in B_{\frac{y_N}{2}}(y)$, using the definition of $w$ and Lemma \ref{lem2}, for $r\in(0,R/7)$, we deduce that 
 \begin{align*}
 	|w(x)| \le C y_N \ \textnormal{osc}_{B^+_{3r/2}} \frac{v}{x_N} 
 	  & \le  C y_N \Big(\frac{r}{R} \Big)^\vsg \left(\textnormal{osc}_{B^+_\frac{3R}{14}} \frac{v}{x_N} + \chi \right) 
 	  &\le  C y_N \Big(\frac{r}{R} \Big)^\vsg \Big(\sup_{B^+_{R/4}} \frac{v}{R} +  \chi  \Big),
 \end{align*}
 where in the last inequality we have used \eqref{eq22}. Now from \cite[(2.11)]{colombo}, we obtain 
 \begin{align}\label{eq16}
 	|w(x)| \le C y_N \;\Big(\frac{r}{R} \Big)^\vsg G^{-1} \left[ |B^+_{R/2}|^{-1} \int_{B^+_{R/2}} G(|\na v|) dx\right] \le C  y_N \;\Big(\frac{r}{R} \Big)^\vsg \Big(\sup_{B^+_{R/2}} |\na v|+  \chi  \Big). 
 \end{align}
 Coupling \eqref{eq15} and \eqref{eq16}, and observing the fact that $\text{osc}_{B^+_{r}} \na v =\text{osc}_{B^+_{r}} \na w \le 2 \sup_{B^+_{r}} |\na w|$, we complete the proof of lemma.
\QED
\end{proof}
\textbf{Proof of Lemma \ref{lem4} continued}: Proof of (ii) and (iv) of the lemma follows from the approximation argument as in \cite[Lemma 5.2]{liebm91}. Indeed, we can construct a sequence of operators $A_{1/j}(\cdot)$, with sufficiently smooth coefficient, converging uniformly to $A(0,\cdot)$. Here the coefficient of $A_{1/j}$ exhibit similar bounds depending only on $\La$, $\nu$, $p$, $\omega$ and $N$. For each $j$ large enough, by standard existence theorem, we get a $C^2$ solution to the problem 
\begin{align*}
 -\text{div} A_{1/j}(\na v_j) =0 \; \mbox{ in } B^+_R(0), \quad v_j=u \; \mbox{ on }\pa B_R^+(0).
\end{align*} 
Then, proof of (ii) follows from Lemma \ref{lem5} and noting the fact that constant $C$ depends only on $N,\ba,\La,p$ together with the convergence of $v_j$ to $v$.  Moreover, a consequence of maximum principle implies that $v_j$ attains its maximum and minimum on the boundary, and thus the same will be true for $v$ also. Since $u\ge 0$, we get $\textnormal{osc } u \le \sup u$ and the fact $u\le C x_N$ completes proof of (iv) of the lemma.\QED

 \textbf{Proof of Theorem \ref{thm5}}: We take $u-v$ as a test function in the weak formulations of \eqref{eq51} and \eqref{eq53}, on account of Lemma \ref{lem4}(iv), (A3) and (A4), we deduce that 
  \begin{align*}
	\int_{B_R^+} \big(A(0, \na u)- A(0,\na v) \big) \na (u-v) &= \int_{B_R^+} \big(A(x, \na u)- A(0,\na v) \big) \na (u-v)  \nonumber\\
	& \ \ + \int_{B_R^+} \big(A(0, \na u)- A(x,\na u) \big) \na (u-v)  \nonumber\\ 
	&\le \int_{B_R^+} \big(|B(x,u,\na u)|+g(x)\big)|u-v| \\
	&\quad+ \La\int_{B^+_R} |x|^\omega \big(1+|\na u|\big)^{p-1} |\na u-\na v| \nonumber \\
	&\le C\left[ \La R \int_{B_R^+} (1+|\na u|)^p+R \int_{B_R^+} x_N^{-\sg}~dx \right.\\
	&\left.\qquad + R^\omega \int_{B^+_R} \big(1+|\na u|\big)^{p-1} |\na u-\na v| \right]. 
  \end{align*}
 Noting the fact that $\sg,\omega,R<1$ and using Young inequality, for $\varepsilon>0$, we obtain
 \begin{align}\label{eq54}
  \int_{B_R^+} \big[A(0, \na u)- A(0,\na v) \big] \na (u-v) &\le c R^{N+1-\sg}	+\varepsilon R^\omega \int_{B^+_R} |\na (u- v)|^p \nonumber\\
  &\quad+ (C_\varepsilon+C) R^\omega  \int_{B^+_R} \big[1+|\na u|\big]^{p}
 \end{align} 
 where $c,C,C_\varepsilon>0$  are constants. We fix the following  
 \begin{align*}
  J(w;R) =\int_{B_R^+} |\na w|^p dx \ \mbox{ and } I(w;R)=  \int_{B_R^+} |\na w-(\na w)_R|^p,
 \end{align*}
 with $(w)_R= \frac{1}{|B_R^+|} \int_{B_R^+} w(x)dx$. 
 Now proceeding similar to \cite[Page 150]{giacomoni}, for $ p\ge 2$, we have
	\begin{align*}
	\big(A(0, \na u)- A(0,\na v) \big) \na (u-v) \ge \nu \kappa_p |\na (u-v)|^p,
	\end{align*}
 where $\kappa_p>0$ is a constant. Therefore, using this in \eqref{eq54}, for suitable $\varepsilon>0$, we obtain
	\begin{align}\label{eq56}
	J(u-v;R)=\int_{B_R^+} |\na (u-v)|^p \le C \left(R^{N+1-\sg}+ R^\omega \int_{B^+_R} \big[1+|\na u|\big]^{p}\right).
	\end{align}
For the case $1<p<2$, we recall the following result of \cite[Lemma 1]{tolksdorf}
 \begin{align*}
  \nu \int_{B^+_R} \big(1+|\na u|+|\na v|\big)^{p-2}|\na(u-v)|^2 \le C \int_{B_R^+} \big[A(0, \na u)- A(0,\na v) \big] \na (u-v). 
 \end{align*}
On the account of Lemma \ref{lem4}(iii) and by repeated application of H\"older and Young inequality, for $\ga>0$, we deduce that 
 \begin{align*}
 J(u-v;R)&\le \left[ \int_{B^+_R} \big(1+|\na u|+|\na v|\big)^{p-2}|\na(u-v)|^2 \right]^\frac{p}{2} \left[\int_{B^+_R} \big(1+|\na u|\big)^p\right]^\frac{2-p}{2} \\
 &\le C\left[R^\frac{-2\ga}{p} \int_{B_R^+} \big[A(0, \na u)- A(0,\na v) \big] \na (u-v) +  R^\frac{2\ga}{2-p} \int_{B^+_R} \big(1+|\na u|\big)^p \right].
 \end{align*}
Using \eqref{eq54} in the above expression, we get
\begin{align*}
 J(u-v;R)&\le C R^\frac{-2\ga}{p} \Big[ c R^{N+1-\sg}	+ \varepsilon R^\omega \int_{B^+_R} |\na (u- v)|^p + (C_\varepsilon +C) R^\omega \int_{B^+_R} \big[1+|\na u|\big]^{p}\Big] \\
 &\qquad+ C R^\frac{2\ga}{2-p} \int_{B^+_R} \big(1+|\na u|\big)^p.
\end{align*}
We choose $\ga>0$ sufficiently small such that $1-\sg-2\ga/p>0$ and $\omega-2\ga/p>0$, and we set $\ga_0=\min\{ 1-\sg-2\ga/p, \omega-2\ga/p, 2\ga/(2-p) \}>0$. Therefore, for sufficiently small $\varepsilon>0$,  we infer that
\begin{align}\label{eq57}
 J(u-v;R)&\le C R^{\ga_0} \Big[ R^N+ \int_{B^+_R} \big(1+|\na u|\big)^p \Big].
\end{align}
Combining \eqref{eq56} and \eqref{eq57}, we obtain
\begin{align}\label{eq58}
 J(u-v;R)=\int_{B^+_R}|\na (u-v)|^p\le C \Big[ R^{N+\ga_0}+  R^{\ga_0}  \big(R^N+J(u;R)\big) \Big], \quad\mbox{ for all }p>1.
\end{align}
 Furthermore, using Lemma \ref{lem4}(i) and (iii), we observe that 
\begin{align}\label{eq59}
 J(v;r)= \int_{B^+_R} |\na v|^p \le C r^N \Big[\sup_{B^+_{R/2}} |\na v|\Big]^p &\le Cr^N\Big[R^{-N} J(v;R)+\chi \Big] \nonumber\\
 &\le C \Big[R^N+\Big(\frac{r}{R}\Big)^{N}J(u;R) \Big].
\end{align}
Taking into account \eqref{eq58} and \eqref{eq59} together with the estimate of $J(u;1)$, due to \cite[Lemma 3]{simon}, a standard procedure yields
	\begin{align}\label{eq55}
	 J(u;R)\le C_\tau R^{N-\tau} \mbox{ for }0<R<1 \mbox{ and any }\tau>0.
	\end{align}
Thus, \eqref{eq58} reduce to 
 \begin{align}\label{eq60}
 	J(u-v;R)\le C \big( R^{N+\ga_0}+  R^{N-\tau+\ga_0} \big), \quad\mbox{ for all }p>1.
 \end{align}
Now, consider
\begin{align*}
 I(u;r) = \int_{B_r^+} |\na u-(\na u)_r|^p &\le C\left[ \int_{B_r^+} |\na (u- v)|^p + \int_{B_r^+} |\na v- (\na v)_r|^p \right] \\
  &\le C \Big[J(u-v;R)+r^N \big(\textnormal{osc}_{B_r^+} \na v\big)^p\Big].
\end{align*}
Then, using \eqref{eq60} and Lemma \ref{lem4}(i)-(iii) in the above expression, we deduce that 
 \begin{align*}
  I(u;r)\le C\left[R^{N-\tau+\ga_0}+ r^N \Big(\frac{r}{R}\Big)^{\vsg p} \big(R^{-N} J(v;R)+\chi \big) \right] \le C\left[ R^{N-\tau+\ga_0}+ r^{N+\vsg p} R^{-\vsg p-\tau} \right],
 \end{align*}
where in the last inequality we have used \eqref{eq55} to estimate $J(u;R)$ and the fact that $0<R<1$. Setting $R=r^\theta$, appropriate choice of $\tau$ and $\theta$ yields
 \begin{align*}
  I(u;r)\le C r^{N+\al p},	
 \end{align*} 
for some positive constant $\al$. It is clear that the constant $\al$ depends only on $N,\omega,\sg,\nu,\La$ and $p$. By the equivalence of Campanato and H\"older norms, and covering argument, we get $u\in C^{1,\al}(\ov\Om)$.\QED

 \textbf{Proof of Theorem \ref{thm6}}: Since $u$ is a weak solution to $(P)$, for all $\phi\in C^\infty_c(\Om)$, the following holds
 \begin{align}\label{eq20}
\int_{\Om} |\na u|^{p-2} \na u\na \phi+ \int_{\Om} |\na u|^{q-2} \na u\na \phi = \int_{\Om} f(x)\;u^{-\de}\phi.
\end{align}
For $\ga>1$, to be chosen later, from  \eqref{eq20}, we infer that
\begin{align*}
  \ga^{1-p}\int_{\Om} u^{(1-\ga)(p-1)} |\na u^\ga|^{p-2} \na u^\ga\na \phi+ \ga^{1-q}\int_{\Om} u^{(1-\ga)(q-1)}|\na u^\ga|^{q-2} \na u^\ga\na \phi = \int_{\Om} f(x)u^{-\de}\phi,
\end{align*}
that is, 
  \begin{align*}
	&\int_{\Om} |\na u^\ga|^{p-2} \na u^\ga\na\big(u^{(1-\ga)(p-1)}\phi \big) +(\ga-1)(p-1) \int_{\Om} u^{(1-\ga)(p-1)-1} \phi|\na u^\ga|^{p-2} \na u^\ga\na u  \nonumber \\
	& +\ga^{p-q}\int_{\Om} \big[|\na u^\ga|^{q-2} \na u^\ga\na\big(u^{(1-\ga)(q-1)}\phi \big) +(\ga-1)(q-1) u^{(1-\ga)(q-1)-1} \phi|\na u^\ga|^{q-2} \na u^\ga\na u \big]  \nonumber
	\\&\qquad  =\ga^{p-1} \int_{\Om} f(x) u^{-\de}\phi.
  \end{align*}
 This implies that,
	\begin{align*}
	 -u^{(1-\ga)(p-1)}\De_p u^\ga - &\ga^{p-q} u^{(1-\ga)(q-1)}\De_q u^\ga+ (\ga-1)(p-1) u^{(1-\ga)(p-1)-1}|\na u^\ga|^{p-2}\na u^\ga \na u \\
	 &+ \ga^{p-q}(\ga-1)(q-1) u^{(1-\ga)(q-1)-1}|\na u^\ga|^{q-2}\na u^\ga \na u = \ga^{p-1} f(x) u^{-\de},
	\end{align*}
 equivalently,
  \begin{align}\label{eq25}
	-\De_p u^\ga - \ga^{p-q} u^{(\ga-1)(p-q)}\De_q u^\ga + (\ga-1)(p-1) \frac{|\na u^\ga|^p}{u^\ga} 
	&+ \ga^{p-q}(\ga-1)(q-1) \frac{|\na u^\ga|^q}{u^{\ga+(\ga-1)(q-p)}}  \nonumber \\
	&  \ \ = \ga^{p-1} f(x) u^{-\de+(\ga-1)(p-1)}.
  \end{align}  	
On the account of assumption $u \le \Ga d^{\tl\sg}$, we see that the right hand side of \eqref{eq25} can be bounded from above by $C(\ga,\Ga,c_2) d^{-\ba+\tl\sg(-\de+(\ga-1)(p-1))}$.
Therefore, we choose $\ga>1$ such that $u^\ga\in W^{1,p}_0(\Om)$ and $-\ba+\tl\sg\big(-\de+(\ga-1)(p-1)\big)>0$ so that the right hand side of \eqref{eq25} is in $L^\infty(\Om)$. For convenience, we denote $u^\ga = v\in W^{1,p}_0(\Om)$, thus \eqref{eq25} takes the following form
 \begin{align}\label{eq26}
  -\De_p v- \ga^{p-q} v^{(\ga-1)(p-q)/\ga}\De_q v + (\ga-1)(p-1) \frac{|\na v|^p}{v} 
  &+ \ga^{p-q}(\ga-1)(q-1)\frac{|\na v|^q}{v^{1+(\ga-1)(q-p)/\ga}} \nonumber\\
  & \quad= \ga^{p-1} f(x) v^{\frac{-\de+(\ga-1)(p-1)}{\ga}}.
 \end{align}
Now, we will prove that $v\in C^{0,\al_1}(\ov\Om)$, for some $\al_1\in(0,1)$. Let us assume that $0\in\Om$ and $\rho>0$ be such that $B_{2\rho}:=B_{2\rho}(0)\Subset\Om$. \\
\textbf{Claim}: There exists a constant $C>0$, depending only on $N,p$ and $\|v\|_{L^\infty(\Om)}$, such that $\int_{B_\rho} |\na v|^p \le C \rho^{N-p}$.\\
We note that the similar result holds if $B_\rho$ is replaced by $\Om\cap B_\rho$, when $B_\rho$ is centered on the boundary of $\Om$. Let $\zeta\in C^\infty_c(\Om)$ such that $0\le \zeta \le 1$, $\zeta=1$ in $B_\rho$ and supp$(\zeta)\subset B_{2\rho}$. We take $\phi= e^{\tau v}\zeta^p$, for some $\tau>0$, in the weak formulation of \eqref{eq26}, thus
\begin{align*}
  \int_{\Om} \Big[ |\na v|^{p-2}\na v \na \phi + \ga^{p-q} |\na v|^{q-2} \na v \na( v^\frac{(\ga-1)(p-q)}{\ga}\phi)+ (\ga-1)(p-1) \frac{|\na v|^p}{v}\phi \\
  + \ga^{p-q}(\ga-1)(q-1)\frac{|\na v|^q}{v^{1+(\ga-1)(q-p)/\ga}} \phi -\ga^{p-1} f(x) v^{\frac{-\de+(\ga-1)(p-1)}{\ga}}\phi \Big]dx=0.
\end{align*} 
Since $\phi\ge 0$ and $v\ge 0$ in $\Om$, we observe that the third and fourth term of the integrand is nonnegative. Therefore,
 \begin{align*}
  \int_{\Om} e^{\tau v}\big[ |\na v|^{p-2}\na v(\tau \zeta^p\na v+ p\zeta^{p-1}\na\zeta)+v^\frac{(\ga-1)(p-q)}{\ga} |\na v|^{q-2}\na v(\tau \zeta^p \na v+ p\zeta^{p-1}\na\zeta)\\
  + \frac{(\ga-1)(p-q)}{\ga}v^{\frac{(\ga-1)(p-q)}{\ga}-1}|\na v|^q \zeta^p-\ga^{p-1} f(x) v^{\frac{-\de+(\ga-1)(p-1)}{\ga}} \zeta^p\big]dx \le 0, 
 \end{align*}
which again due to the nonnegativity of $\zeta$ and $v$, and for $\ga>1$ with $(\ga-1)(p-q)>\ga$, implies
\begin{align*}
  \tau \int_{\Om} e^{\tau v} |\na v|^{p} \zeta^p dx \le C\int_{\Om} e^{\tau  v} \big[ |\na v|^{p-1} p\zeta^{p-1}|\na\zeta| + \|v\|_{L^\infty(\Om)}^\frac{(\ga-1)(p-q)}{\ga} |\na v|^{q-1} p\zeta^{p-1}|\na\zeta|\\
  +\ga^{p-1} d^{-\ba+\tl\sg(-\de+(\ga-1)(p-1))} \zeta^p\big]dx.
\end{align*}
Then, proof of the claim can be completed proceeding similar to \cite[Lemma 1.1, p. 247]{ladyzh}. Employing \cite[Theorem 1.1, p.251]{ladyzh}, we conclude that $v=u^\ga\in C^{0,\al_1}(\bar\Om)$, here $\al_1\in(0,1)$ is a constant depending only on $\|u\|_{L^\infty(\Om)}$, $N,p$. Therefore, we infer that $u\in C^{0,\al}(\bar\Om)$, where $\al= \al_1/\ga$. This completes proof of the theorem.\QED
\begin{Remark}
 For the case $\ba< 2-1/p$ and $\Om = B_R(0)$, by standard procedure, we get the existence of a unique radial solution $u$ of $(P)$. Moreover, similar arguments as in the proof of Theorem  \ref{thm5} imply $u^\ga$ in $C^1(\ov\Om)$. In this manner the regularity results of Theorem \ref{thm6} and Corollary \ref{cor1} are improved.
\end{Remark}
 
\noi\textbf{Proof of Corollary \ref{cor1}}: We note that either by the uniqueness result or due to the minimality of the solution, on account of Theorem \ref{thm1}, we get $u\in \mc C_{d_{\beta,\delta}}$. Thus, for the case of $\beta+\delta<1$, Theorem \ref{thm5} ensures that $u\in C^{1,\al}(\ov\Om)$, whereas for the case $\ba+\de \ge 1$, Theorem \ref{thm6} implies $u\in C^{0,\al}(\ov\Om)$. This completes proof of the corollary.\QED

\begin{Remark}
We remark that our H\"older continuity result of Theorem \ref{thm5} for the gradient of weak solution, in the case of $\ba+\de<1$, can be used to obtain multiplicity results for the following quasilinear elliptic equation involving singular nonlinearity 
 \begin{equation*}
 (S) \; \left\{\begin{array}{rllll}
 \;  -\Delta_{p}u -\Delta_{q} u  &= f(x) u^{-\delta}+g(x,u),\; u>0 \text{ in }\; \Om,\\
  u&=0 \text{ on }  \pa\Om,
 \end{array}
 \right.
 \end{equation*}
 where $g(x,t)$ is a Carath\'eodory function satisfying 
 \begin{enumerate}
 \item[($g_1$)] $g(x,t)\ge 0$ for all $(x,t)\in\ov\Om\times\mb R^+$ with $g(x,0)=0$.
 \item[($g_2$)] There exists $r>p-1$ with $r\le  p^*-1:= \frac{Np}{N-p}-1 $, if $p<N$, otherwise $r<\infty$
 such that $g(x,t)\le C(1+t)^r$ for all $(x,t)\in\Om\times\mb R^+$, for some constant $C>0$.
\end{enumerate}
 We define the associated energy functional $\mc I: W^{1,p}_0(\Om)\to\mb R$ as 
 \begin{align*}
 \mc I(u):=  \frac{1}{p} \int_{\Om}|\na u|^p dx+\frac{1}{q} \int_{\Om} |\na u|^q dx- \frac{1}{1-\de}\int_{\Om} f(x) |u|^{1-\de} dx-\int_{\Om} G(x,u)dx,
 \end{align*}
 where $G(x,u)=\int_{0}^{u} g(x,t)dt$.
 Following the approach in \cite{giacomoni1}, we can prove the following Sobolev versus H\"older minimizer result.
 \begin{Theorem}
 Let $u_0\in C^1(\ov\Om)$ satisfying $u_0\ge \eta d(x)$ in $\Om$, for some $\eta>0$, be a local minimizer of $\mc I$ in the $C^1(\ov\Om)\cap C_0(\ov\Om)$ topology. Then, $u_0$ is a local minimizer of $\mc I$ in $W^{1,p}_0(\Om)$ topology also.
 \end{Theorem}
Using the above theorem and strong comparison principle for singular problems we can prove the existence of one positive solution $u_1$ to $(S)$. Consequently, using critical point analysis similar to \cite[Section 5]{dpk}, we obtain the existence of second solution.
\end{Remark}

\begin{Remark}
 We remark that Theorem \ref{thm6} can be used to obtain the existence results for quasilinear elliptic systems driven by the nonhomogeneous $p-q$ Laplace operator and involving singular nonlinearities by using Schauder fixed point theorem similar to \cite{giacomoni2}. 
\end{Remark}

\medskip
{\bf Acknowledgments}\\
 The authors thank the anonymous referee for the careful reading of this manuscript and for his/her remarks and comments, which have improved the initial version of our work. D. Kumar is thankful to Council of Scientific and Industrial Research (CSIR) for the financial support. K. Sreenadh acknowledges the support through the Project:
  MATRICS grant MTR/2019/000121 funded by SERB, India.

\end{document}